\documentclass[10pt,%
               			a4paper,%
               			oneside,%
					reqno,
					english
               			]{smfart}
                                          
\usepackage{amssymb}
\usepackage{tensor,slashed}
\usepackage{physics}
\usepackage{smfthm}

\usepackage[all,cmtip]{xy}
\usepackage{tikz}
\usetikzlibrary{cd}
\usepackage{xparse}
\usepackage{stackengine}		
\usepackage{scalerel}

\usepackage[T1]{fontenc}
\usepackage{lmodern}
\usepackage[utf8]{inputenc}

\usepackage[english]{babel}
\usepackage{microtype}
\usepackage{hyperref}
\DeclareRobustCommand{\SkipTocEntry}[5]{}

\newcommand{\numberset}{\mathbb} 
\newcommand{\N}{\numberset{N}}

\let\S\relax
\let\K\relax
\let\P\relax
\let\F\relax

\newcommand{\F}{\mathcal{F}}
\newcommand{\T}{\mathcal{T}}
\newcommand{\K}{\mathcal{K}}
\newcommand{\S}{\mathcal{S}}
\newcommand{\Nn}{\mathcal{N}}
\newcommand{\P}{\mathcal{P}}

\newcommand{\Cc}{\mathcal{C}}
\newcommand{\I}{\mathcal{I}}

\let\E\relax
\newcommand{\E}{\numberset{E}}

\newcommand{\cK}{\mathcal{K}}
\newcommand{\cA}{\mathcal{A}}

\newcommand{\cE}{\mathcal{E}}
\newcommand{\cL}{\mathcal{L}}
\newcommand{\EE}{\cE}
\newcommand{\EF}{\mathfrak{E}}
\newcommand{\AF}{\cA}

\newcommand{\KK}{\cK}
\newcommand{\LL}{\cL}
\newcommand{\GG}{G}

\newcommand{\PI}{\langle P_\I\rangle}
\newcommand{\CI}{\langle \Cc\I\rangle}
\newcommand{\hlim}{\text{ho-lim}\,}

\DeclareMathOperator{\RKK}{RKK}
\DeclareMathOperator{\KKK}{KK}
\DeclareMathOperator{\Hom}{Hom}
\DeclareMathOperator{\Cone}{Cone}
\DeclareDocumentCommand{\KK}{O{} O{} m m}{\KKK_{#1}^{#2}(#3,#4)}

\newcommand{\hatotimes}{\,\widehat{\otimes}\,}
\newcommand{\hot}{\widehat{\otimes}}

\let\E\relax
\let\Res\relax
\newcommand{\E}{\underline{E}}

\DeclareMathOperator{\Ind}{Ind}

\DeclareMathOperator{\Res}{Res}

\newcommand{\tens}[3]{%
  \mathbin{\mathop{\tensor[_#1]{\otimes}{_#3}}\displaylimits_{#2}}%
}

\newtheorem{theoremA}{Theorem}

\author{Christian Bönicke}

\address{School of Mathematics and Statistics, University of Glasgow, University Place, Glasgow G12 8QQ, United Kingdom}
\email{christian.bonicke@glasgow.ac.uk}
\curraddr{\textit{Current Address:} School of Mathematics, Statistics and Physics, Newcastle University, Newcastle upon Tyne NE1 7RU, United Kingdom}
\email{christian.bonicke@ncl.ac.uk}

\author{Valerio Proietti}
\address{Graduate School of Mathematical Sciences, The University of Tokyo}
\email{valerio@ms.u-tokyo.ac.jp}

\title[Baum--Connes conjecture for étale groupoids]{Categorical approach to the Baum--Connes conjecture for étale groupoids}

\begin{document}
\frontmatter

\begin{abstract}
We consider the equivariant Kasparov category associated to an étale groupoid, and by leveraging its triangulated structure we study its localization at the ``weakly contractible'' objects, extending previous work by R.~ Meyer and R.~Nest. We prove the subcategory of weakly contractible objects is complementary to the localizing subcategory of projective objects, which are defined in terms of ``compactly induced'' algebras with respect to certain proper subgroupoids related to isotropy. The resulting ``strong'' Baum--Connes conjecture implies the classical one, and its formulation clarifies several permanence properties and other functorial statements. We present multiple applications, including consequences for the Universal Coefficient Theorem, a generalized ``Going-Down'' principle, injectivity results for groupoids that are amenable at infinity, the Baum-Connes conjecture for group bundles, and a result about the invariance of $K$-groups of twisted groupoid $C^*$-algebras under homotopy of twists.
\end{abstract}


\maketitle
\setcounter{tocdepth}{2}
\tableofcontents

\mainmatter

\addtocontents{toc}{\SkipTocEntry}\section*{Introduction and main results}

Over the last decades étale groupoids and their homological and $K$-theoretical invariants have played an increasingly important role in the fields of operator algebras, noncommutative geometry and topological dynamics.
Kumjian and Renault showed that $C^*$-algebras associated with groupoids provide versatile models for large classes of $C^*$-algebras \cite{kum:diagonals,ren:cartan}. More recently, Li showed that every classifiable $C^*$-algebra admits a (twisted) groupoid model \cite{li:cartan}. One of the biggest open questions in the field concerns the Universal Coefficient Theorem (UCT) and work of Barlak and Li \cite{bali:CartanUCT} showed that the UCT problem can be translated to the question whether every nuclear $C^*$-algebra admits a groupoid model.

In another direction, Matui's works \cite{matui:hk,matui:productSFT} have kickstarted a fruitful line of research in topological dynamics using étale groupoids at its heart (see also \cite{li:orbitequivalence}). In this area it turns out that many invariants for topological dynamical systems can most naturally been defined in the framework of groupoid homology or the $K$-theory of groupoid $C^*$-algebras. 
Consequently, there is a great deal of interest around the homology and $K$-theory of \'etale groupoids and their interaction. Examples of recent research in this direction are the HK conjecture of Matui \cite{matui:hk}, or the relation between the homology theory of Smale spaces and the $K$-theory of their corresponding $C^*$-algebras \cite{put:HoSmale}. In this latter example, a special case of the methods developed here (i.e., when the groupoid is torsion-free and ample) has already been applied with great success and lead to many interesting results in topological dynamics, as is demonstrated by the papers \cite{bdgw:matui, valmak:groupoid, valmak:groupoidtwo, valmak:groupoidthree}.

Motivated by these developments we set out to develop the category-theory based approach to the Baum--Connes conjecture for the class of étale groupoids in full generality. This approach  is very suitable for formulating and proving general statements about the Baum--Connes conjecture, and for obtaining functorial properties of the assembly map and $K$-theoretic duality type results \cite{meyereme:dualities,val:shi}. As already observed by Meyer and Nest \cite{nestmeyer:loc}, many permanence results of the Baum--Connes conjecture become quite accessible in this setup. Besides this, several results obtained by the first named author \cite{bon:goingdownbc, B21} and C.~Dell'Aiera \cite{bondel:kun} are generalized to all \'etale groupoids.

The following statement summarizes a selection of applications that we are able to obtain through this approach. Some statements are deliberately vague to spare the reader the technical details at this stage, we refer to the final section of this article (Section \ref{sec:app}) for the definitions and more precise statements.

\begin{theoremA}\label{thm:0}
Let $G$ be an \'etale groupoid which is second countable, locally compact, and Hausdorff.
\begin{enumerate}
    \item Suppose $\Sigma$ is a twist over $G$. If $G$ satisfies the strong Baum-Connes conjecture, then $C_r^*(G,\Sigma)$ satisfies the UCT.
    \item The $K$-theory of $C_r^*(G,\Sigma)$ only depends on the homotopy class of $\Sigma$.
    \item If $G$ is strongly amenable at infinity, then there is a dual Dirac morphism for $G$. In particular, the Baum-Connes assembly map is split-injective.
    \item The (strong) Baum-Connes conjecture enjoys many permanence properties both with respect to the involved groupoid (it passes to subgroupoids, direct products, increasing unions) and the coefficient algebra (inductive limits, tensor products).
\end{enumerate}
  
\end{theoremA}

The results in Theorem \ref{thm:0} should be compared to another line of research, which uses quantitative $K$-theory methods to obtain many interesting related results on the UCT, the Baum-Connes conjecture and its permanence properties \cite{GWY23,OO23,WY23}. 

In \cite{nestmeyer:loc} R.~Meyer and R.~Nest established the category theoretic framework we are after in the setting of locally compact groups, and more generally for transformation groups. To this end, they leverage the triangulated structure of the equivariant bivariant Kasparov category, and in particular the notion of complementary subcategories and localization. This paper extends these methods to include étale groupoids.

A related approach is described in \cite{luckdavis:fjbc}, where the authors give a unified approach to various isomorphism conjectures, including the Baum--Connes conjecture, by means of the orbit category and the homotopy theory of spectra. In both approaches, the role of \emph{weakly} contractible objects, defined in terms of a certain family of subgroups of a given group $G$, is in a certain sense fundamental. For the Baum--Connes conjecture associated to a discrete group, this family is given by the finite subgroups of $G$. Analogously, when $G$ is locally compact, the family is given by the compact subgroups.

Thus the first task when attempting to generalise this approach is the identification of a suitable class of subgroupoids of a given étale groupoid $G$. Associated to this class is a homological ideal in the Kasparov category $\KKK^G$, which is the starting point for several notions of \emph{relative} homological algebra, e.g., the notion of projective object. In the words of Meyer and Nest \cite[page 215]{nestmeyer:loc}, ``it is not so clear what should correspond to compact subgroups'' in the case of the Baum--Connes conjecture for groupoids.

A partial solution to this question was offered in \cite{meyereme:dualities}, where the authors show a relation of complementarity between the subcategory of proper objects and the objects $A\in \KKK^G$ such that $p^*(A)$ is contractible in $\KKK^{G\ltimes \E G}$. Here $p^*$ is the pullback functor associated to the projection $p\colon G\ltimes\E G\to G$ where $\E G$ denotes the universal example for proper actions (which is well-defined for groupoids, see for example \cite{tu:novikov}). This approach is based on the fact that $p^*$ is, effectively speaking, the localization functor which we seek (see Theorem \ref{thm:rkkloc}). However, this is not completely satisfactory because (\textbf{a}) it relies on the existence of a Kasparov dual \cite[Theorem 4.37]{meyereme:dualities}, and (\textbf{b}) it does not present the projective objects in terms of a simpler class of ``building blocks'' constructed via induction on a suitable family of subgroupoids.

This paper remedies these shortcomings by using a ``slice theorem'' (see Proposition \ref{prop:prodr} below and compare with \cite[Proposition 2.42]{tu:propnonhaus}) for étale groupoids acting properly on a space, which allows us to identify a family of subgroupoids that we call ``compact actions'', as they are isomorphic to action groupoids for finite groups sitting inside the isotropy of $G$. On a first approximation, we can say that the family of compact subgroups is replaced in our case by the family of \emph{proper} subgroupoids of $G$ (see Lemma \ref{lem:contractible objects} for more details on this statement).

Having this, most of the machinery from \cite{nestmeyer:loc} can be reproduced in the groupoid context in a straightforward fashion, as it is mostly formal and inherited from the more general theory of triangulated categories. We say ``most'' because we encountered another technical difficulty along the way, which we now briefly explain. Having defined projective objects as retracts of (direct sums of) ``compactly induced'' objects, we were facing the issue of identifying the localizing subcategory of proper objects with the one induced by projectives. Indeed, a result of this kind is highly desirable because not only it would match up nicely with the statement in \cite{meyereme:dualities}, but more importantly it allows to rephrase the main result of \cite{tu:moy}, on the Baum--Connes conjecture for a groupoid $G$ satisfying the Haagerup property, as a proof that the category $\KKK^G$ is generated by projective objects as defined by us.

A blueprint for this result ought to be found in \cite{nestmeyer:loc}, and indeed \cite[Theorem 7.1]{nestmeyer:loc} and its applications correspond to the statement we need. Nevertheless, we were not able to simply generalize the proof therein, essentially because (\textbf{a}) our compact actions are \emph{open} subgroupoids, and (\textbf{b}) the excisive properties of $\RKK^G(-\,;A, B)$ are not entirely clear (at least to us) in general, even in simple cases such as homotopy pushouts. Nevertheless, by briefly passing to $E$-theory (which has long exact sequences without extra hypotheses) and using the fact that localizing subcategories are closed under direct summands, we are able to find an alternative proof of the identification of localizing subcategories of (respectively) compactly induced and proper objects.

Before passing to the organization of the paper, we present two of the core results which should serve as a brief summary of this work. For more details on definitions and applications the reader should consult sections \ref{sec:prelims} and \ref{sec:app}.

\begin{theoremA}\label{thm:a}
Let $\mathcal{N}\subseteq \KKK^G$ be the subcategory of $G$-$C^*$-algebras $A$ such that $\Res^G_H(A)\cong 0$ for any proper open subgroupoid $H\subseteq G$. Let $\mathcal{P}\subseteq \KKK^G$ be the smallest localizing triangulated subcategory containing proper $G$-$C^*$-algebras. Then $(\mathcal{P},\mathcal{N})$ is a pair of complementary subcategories and $\mathcal{P}$ is generated by ``\emph{compactly induced}'' objects (see Theorem \ref{thm:bccc} for details).
\end{theoremA}

The previous result implies that, for any $A\in\KKK^G$, there is an exact triangle, functorial in $A$ and unique up to isomorphism, such that $P(A)\in\mathcal{P}$ and $N(A)\in \mathcal{N}$,
\[
\Sigma N(A)\longrightarrow P(A) \longrightarrow A \longrightarrow N(A).
\]
Following \cite{meyer:tri}, the object $P(A)$ is called the \emph{cellular approximation} of $A$. We should point out that if $P(C_0(G^0))$ is a \emph{proper} $G$-$C^*$-algebra, then any $A\in \mathcal{P}$ is $\KKK^G$-equivalent to a proper $C^*$-algebra (see Remark \ref{rem:propsub}).

The next result gives a more familiar presentation of the localization $\KKK^G/\mathcal{N}$, and expresses the ordinary Baum--Connes conjecture in terms of the natural morphism $D_A\colon P(A)\to A$ introduced above. We can view this theorem as a bridge between the somewhat abstract notions arising via the triangulated category approach and more classical objects, such as the $\RKK$-group and the ``topological'' $K$-theory group appearing at the left-hand side of the Baum--Connes conjecture.

\begin{theoremA}\label{thm:b}
Let $p\colon \E G\to G^{0}$ be the structure map of the $G$-action. The pullback functor descends to an isomorphism of categories $p^*\colon \KKK^G/\mathcal{N}\to \RKK(\E G)$. The induced map $(D_A\rtimes_r G)_*\colon K_*(P(A)\rtimes_r G) \to K_*(A\rtimes_r G)$ corresponds to the assembly map under the natural identification $K_*^{\mathrm{top}}(G;A)\cong K_*(P(A)\rtimes_r G)$.
\end{theoremA}

The paper is organized as follows. In Section \ref{sec:prelims} we lay out the fundamental definitions and conventions which we use throughout the paper. We define groupoid crossed products, pass on to discussing the triangulated structure of the equivariant $\KKK$- and $E$-categories, and finish with some basic results on complementary subcategories and homotopy direct limits. Section \ref{sec:indres} is entirely dedicated to the main technical result of the paper, i.e., an adjunction between the functors $\Ind_H^G\colon \KKK^H\rightleftarrows  \KKK^G\colon \Res_G^H$.

This adjoint situation is the technical foundation for the main results of the paper. Its proof is fairly complicated in terms of bookkeeping of variables, but it does not require particularly new conceptual ideas. In fact, the definition for unit and counit are very intuitive in terms of the open inclusion $H\subseteq G$. The model for the induction functor is perhaps a minor point of novelty, as it is based on the crossed product construction rather than on (generalized) fixed-point algebras. This is especially useful as an open subgroupoid $H\subseteq G$ need not act on $G$ properly (see Remark \ref{rem:indmodel}).

Section \ref{sec:bc} is entirely dedicated to proving Theorem \ref{thm:a} and \ref{thm:b} above, along with some other auxiliary results. The excisive properties of $E$-theory are used in this section.

Section \ref{sec:app} discusses several applications of the main results of the paper. In particular, we give the precise statements and proofs of the results mentioned in Theorem \ref{thm:0}.

\addtocontents{toc}{\SkipTocEntry}\section*{Acknowledgements}

We would like to thank R. Meyer, R. Nest, and M. Yamashita for many helpful suggestions. We are also grateful to A. Miller and S. Nishikawa for pointing out errors in a previous version of this manuscript.

The first author was supported by the Alexander von Humboldt Foundation.

The second author was supported by: Science and Technology Commission of Shanghai Municipality (grant No.~18dz2271000), Foreign Young Talents' grant (National Natural Science Foundation of China), CREST Grant Number JPMJCR19T2 (Japan), Marie Skłodowska-Curie Individual Fellowship (project number 101063362).

\section{Preliminaries}\label{sec:prelims}

Let $G$ be a second countable, locally compact, Hausdorff groupoid with unit space $G^{0}$. We let $s,r\colon G\to G^0$ denote respectively the source and range maps. In addition, we use the notation $G_x=s^{-1}(x)$, $G^x=r^{-1}(x)$, and for a subset $A \subset G^{0}$, we write $G_A = \bigcup_{x \in A} G_x$, $G^A = \bigcup_{x \in A} G^x$, and $G|_A = G^A \cap G_A$. Throughout this paper we assume the existence of a (left) Haar system $\{\lambda^x\}_{x\in G^0}$ on $G$ \cite{ren:group}. 

Let $X$ be second countable, locally compact, Hausdorff space. A \emph{$C_0(X)$-algebra} is a $C^*$-algebra $A$ endowed with a nondegenerate $*$-homomorphism from $C_0(X)$ to the center of the multiplier algebra $\mathcal{M}(A)$. For an open set $U \subset X$, we define $A_U = C_0(U) A $. For a locally closed subset $Y \subset X$ (i.e., $Y = U \smallsetminus V$ for some open sets $U, V \subset X$), we set $A_Y = A_U / A_{U \cap V}$, and we put $A_x = A_{\{x\}} = A / A C_0(X \setminus \{x\})$ for $x \in X$. More on $C_0(X)$-algebras can be found in \cite{bla:defhopf}.

Let us fix our preliminary conventions on tensor products. A more in-depth discussion is provided after Definition \ref{def:eth}. If $A$ and $B$ are $C_0(X)$-algebras, their \emph{maximal} tensor product $A\otimes B$ is naturally equipped with a $C_0(X\times X)$-structure, and we define the (maximal) balanced tensor product $A\otimes_X B$ as the $C_0(X)$-algebra $(A\otimes B)_{\Delta_X}$, where $\Delta_X\subseteq X\times X$ is the diagonal subspace.

Note that if $f \colon Y \to X$ is a continuous map, then $C_0(Y)$ is a $C_0(X)$-algebra. It is  a continuous field if and only if $f$ is open \cite{blakirch:glimm}. In particular this applies to the situation $Y = G$ and $f = s$, because the source and range maps are open when a Haar system exists \cite[Proposition 2.4]{ren:group}. The map $f$ defines a ``forgetful'' functor, sending a $C_0(Y)$-algebra $A$ to a $C_0(X)$-algebra $f_*(A)$, by way of the composition $C_0(X)\to \mathcal{M}(C_0(Y))\to Z\mathcal{M}(A)$. In addition, for a $C_0(X)$-algebra $B$, a continuous function like $f$ above also induces a \emph{pullback} functor $f^*B = C_0(Y) \otimes_{X} B$ from the category of $C_0(X)$-algebras to that of $C_0(Y)$-algebras. 

We are ready to define the notion of groupoid action on $C^*$-algebras.
\begin{defi}
Let $G$ be a second countable locally compact Hausdorff groupoid, and put $G^{0} = X$.
A \emph{continuous action} of $G$ on a $C_0(X)$-algebra $A$ (with structure map $\rho$) is given by an isomorphism of $C_0(G)$-algebras
\[
\alpha\colon C_0(G) \tens{s}{X}{\rho} A \to C_0(G) \tens{r}{X}{\rho} A
\]
such that the induced homomorphisms $\alpha_g \colon A_{s(g)} \to A_{r(g)}$ for $g \in G$ satisfy $\alpha_{g h} = \alpha_g \alpha_h$.
In this case, we say that $A$ is a $G$-$C^*$-\emph{algebra}.
\end{defi}

If $A$ is a commutative $C^*$-algebra, say $A\cong C_0(Z)$, then we view the moment map as a continuous function $\rho\colon Z\to X$. In this case, the action $\alpha$ can be given as a continuous map making the following diagram commute,
\[
\xymatrix{G\tensor[_s]{\times}{_\rho}Z  \ar[dr]^-{r} \ar[r]^-{\alpha} & Z \ar[d]^-{\rho}\\
& X}
\]
(above, we are slightly abusing notation by writing $r$ for the map $(g,z)\mapsto r(g)$). The action groupoid obtained this way will be denoted $G\ltimes Z$, it has unit space $Z$ and its generic arrow is determined by a pair $(g,z)\in G\times Z$ with range $z$ and source $\alpha(g^{-1},z)$.

Details on the construction of groupoid crossed product $C^*$-algebras can be found in \cite{skan:crossinv,damu:renequiv}. We are going to only briefly recap the definitions here. Given a $G$-algebra $A$, define the auxiliary algebra $A_0=C_c(G)\cdot r^*A$ and the $\ast$-algebra structure
\begin{align*}
(f\star g)(\gamma)&=\int f(\eta)\alpha_\eta(g(\eta^{-1}\gamma))\,d\lambda^{r(\gamma)}(\eta)\\
f^*(\gamma)&=\alpha_\gamma(f(\gamma^{-1})^*)
\end{align*}
for $f,g\in A_0$. For $f\in A_0$, we also define $\lVert f \rVert_1$ to be the supremum, over $x\in X$, of the quantity $\max\{\int \lVert f(\gamma) \rVert\,d\lambda_x(\gamma), \int\lVert f(\gamma) \rVert\,d\lambda^x(\gamma)\}$, where $\lambda_x(\gamma)=\lambda^x(\gamma^{-1})$. The enveloping $C^*$-algebra of the Banach $\ast$-algebra obtained by completing $A_0$ with respect to $\lVert \cdot \rVert_1$ is called the \emph{full} crossed product of $A$ by $G$. 

In this paper, unless otherwise stated, we are going to consider the \emph{reduced} crossed product $C^*$-algebra of $A$ by $G$, denoted $A\rtimes_r G$ (at times we might drop the subscript ``$r$''), which is obtained as a quotient of the full crossed product as follows. For $x\in X$, consider the $A_x$-Hilbert module $L^2(G^x,\lambda^x)\otimes A_x$. The formula $\Lambda_x(f)g=f\star g$ defines an adjointable operator and extends to a $\ast$-representation of the full crossed product.
\begin{defi}
The \emph{reduced} crossed product $A\rtimes_r G$ is defined as the quotient of the full crossed product by the joint kernel of the family $(\Lambda_x)_{x\in X}$ of representations.
\end{defi}

Let us consider the $G$-equivariant Kasparov category $\KKK^G$ whose objects are separable and trivially graded $C^*$-algebras equipped with an action of $G$ and whose set of morphisms $A\rightarrow B$ is Le Gall's groupoid equivariant Kasparov group $\KKK^G(A,B)$ (see \cite{gall:kk}); the composition in this category is the Kasparov product.  We can view $\KKK^G$ as a functor from the category of (separable) $G$-$C^*$-algebras sending equivariant $*$-homomorphisms $A\to B$ to their respective class in the abelian group $\KKK^G(A,B)$. When viewed in this way, the functor $\KKK^G$ enjoys an important property: it is the universal split-exact, $C^*$-stable, and homotopy invariant functor (see \cite{meyer:genhom,valmak:groupoid,thomsen:univ} for more details).

Given a $G$-action on a space $Z$ with moment map $p_Z\colon Z \to G^{0}$, we have introduced above the pullback functor $p^*_Z$ sending $G$-$C^*$-algebras to $G\ltimes Z$-$C^*$-algebras. Thanks to the universal property discussed above, we can promote this functor to a functor between equivariant Kasparov categories $p^*_{Z}\colon \KKK^G\to \KKK^{G\ltimes Z}$. This will be particularly useful when we take $Z$ to be a model for the classifying space for proper actions of $G$ (and in this case we may use the notation $Z=\E G$) \cite[Proposition 6.15]{tu:novikov}.

Moreover, given a map $f\colon G\to G\ltimes Z$, the universal property ensures $f_*$ yields well-defined functor between the corresponding $\KKK$-categories. Furthermore, when $f\colon X\to Z$ is \emph{proper}, we have a standard adjunction (see \cite{nestmeyer:loc})
\begin{equation}\label{eq:stdadj}
\KKK^{G\ltimes X}(f^*A,B)\cong  \KKK^{G\ltimes Z}(A,f_*B).
\end{equation}

Finally, let us define the category $\RKK(Z)$ as follows. 

\begin{defi}
The category $\RKK^G(Z)$ has the same objects as $\KKK^G$, and its $\Hom$-sets $\Hom(A,B)$ are given by the abelian groups $\KKK^{G\ltimes Z}(p_Z^*A,p_Z^*B)$.
\end{defi}

For a map $f$ as above (not necessarily proper), the functor $f^*\colon \KKK^{G\ltimes Z}\to \KKK^{G\ltimes X}$ induces natural maps (slightly abusing notation)
\[
f^*\colon \RKK^G(Z;A,B)\to \RKK^G(X;A,B)
\]
whenever the factorization $p_Z\circ f = p_Y$ holds. In this sense, for fixed $A$ and $B$, $\RKK^G$ is a contravariant functor. It is also homotopy invariant, that is, $f_1^*=f_2^*$ if the maps $f_1, f_2$ are $G$-homotopic. In order to see this, note that we have an isomorphism 
\begin{equation}\label{eq:homiso}
\RKK^G(Y\times [0,1]; A,B)\simeq \RKK^G(Y; A,B[0,1])
\end{equation}
induced by Equation \eqref{eq:stdadj}, hence the claim follows from the homotopy invariance of $\KKK^G(A,B)$ in the second variable $B$.

\subsection{Triangulated structure and comparison with \texorpdfstring{$E$}{E}-theory}

Let us start by fixing some standard conventions. For a $C^*$-algebra $A$, we have a suspension functor $\Sigma A$ defined as $\Sigma A =C_0(\mathbb{R})\otimes A$. For an equivariant $*$-homomorphism of $G$-C$^*$-algebras $f \colon A \to B$, we define its associated mapping cone by
\[
\Cone(f) = \{(a, b_*) \in A \oplus C_0((0, 1], B) \mid f(a) = b_1 \}.
\]
This inherits a structure of $G$-C$^*$-algebra from $A$ and $B$.

An \emph{exact triangle} in $\KKK^G$ is the data of a diagram of the form
\[
A \to B \to C \to \Sigma A,
\]
and a $*$-homomorphism $f\colon A' \to B'$ of $G$-C$^*$-algebras, together with a commutative diagram
\[
\begin{tikzcd}
A \arrow[r] \arrow[d] & B \arrow[r] \arrow[d] & C \arrow[r] \arrow[d] & \Sigma A  \arrow[d]\\
\Sigma B' \arrow[r] & \Cone(f) \arrow[r] & A' \arrow[r] & B',
\end{tikzcd}
\]
where the vertical arrows are equivalences in $\KKK^G$, and the rightmost downward arrow is equal to the leftmost downward arrow, up to applying $\Sigma$ and the Bott periodicity isomorphism $\Sigma^2 B' \simeq B'$ in $\KKK^G$.

As we see from above, the most natural triangulated structure lives on the opposite category $(\KKK^G)^\text{op}$. The opposite category of a triangulated category inherits a canonical triangulated category structure, which has ``the same'' exact triangles. The passage to opposite categories exchanges suspensions and desuspensions and modifies some sign conventions. Thus the functor $\Sigma$ becomes in principle a \emph{desuspension} functor in $\KKK^G$, but due to Bott periodicity $\Sigma$ and $\Sigma^{-1}$ agree, so that we can safely overlook this fact. Moreover, depending on the definition of triangulated category, one may want the suspension to be an equivalence or an isomorphism of categories. In the latter case $\KKK^G$ should be replaced by an equivalent category (see \cite[Section 2.1]{nestmeyer:loc}). This is not terribly important and will be ignored in the sequel.

The triangulated category axioms are discussed in greater detail in \cite{nee:tri,ver:catder}. Most of them amount to formal properties of mapping cones and mapping cylinders, which can be shown in analogy with classical topology. The fundamental axiom requires that any morphism $A \to B$ should be part of an exact triangle. In our setting this can be proved as a consequence of the generalization of \cite{meyer:genhom} to groupoid-equivariant $\KKK$-theory (see also \cite[Lemma A.3.2]{laff:kkban}). Having done that, the rest of the proof follows the same outline of \cite[Appendix A]{nestmeyer:loc}, where the triangulated structure is established in the case of action groupoids. 

There is an alternative, perhaps more conceptual path which consists in \emph{defining} the Kasparov category as a certain localization of the Spanier-Whitehead category associated to the standard tensor category of $G$-$C^*$-algebras and $*$-homomorphisms \cite{ivo:thesis}. The triangulated structure of the Spanier-Whithead category is proved in \cite[Theorem A.5.3]{ivo:thesis}. The argument given there can be directly used to show that $\KKK^G$ is triangulated, because it makes use of only two facts, which we prove below.

\begin{prop}
Let $C$ be the standard tensor category of separable $G$-$C^*$-algebras (with $\otimes_X$) and $*$-homomorphisms. Denote by $F$ the canonical functor from $C$ to $\KKK^G$. The following hold:
\begin{itemize}
\item up to an isomorphism of morphisms in $\KKK^G$, each morphism of $\KKK^G$ is in the image of $F$;
\item up to an isomorphism of diagrams $Q\to K\to D$ in $\KKK^G$, each composable pair of morphisms of $\KKK^G$ is in the image of $F$.
\end{itemize}
\end{prop}
\begin{proof}
In order to show the lifting properties above we make use of ``extension triangles''. Let $f\in \KK[0][G]{Q}{K}$ be a morphism and denote by $\tilde{f}$ the corresponding element $\tilde{f}\in\KK[1][G]{\Sigma  Q}{K}$. By applying \cite[Lemma A.3.4]{laff:kkban} we can represent $\tilde{f}$ by a Kasparov module where the operator $T$ is $G$-equivariant. Then the proof of \cite[Lemma A.3.2]{laff:kkban} gives that $\tilde{f}$ is represented by an equivariant (semi-)split extension which fits a diagram as follows (see \cite[Section 2.3]{nestmeyer:loc}):
\[
\xymatrix{\Sigma^2 Q \ar[r]^-{f\beta_Q^{-1}} \ar@{=}[d] & K \ar[d]^-{\epsilon_K} \ar[r] & E \ar@{=}[d]\ar[r]^-{p_f} & \Sigma Q\ar@{=}[d]\\
\Sigma^2 Q \ar[r]^-{\iota_f} & \Cone(p_f) \ar[r] & E \ar[r]^-{p_f} & \Sigma Q,}
\]
where $\beta_Q$ is the Bott isomorphism and $\epsilon_K$ is an equivalence. Hence we have that $F(\iota_f)\cong f$. Notice how this argument automatically shows that $f$ is contained in an exact triangle (up to equivalence).

Now given $g\in \KK[0][G]{K}{D}$, set $h=g\circ \epsilon_K^{-1},C_f=\Cone(p_f)$ and consider the diagram
\[
\xymatrix{Q \ar[r]^-{f}\ar[d]_-{\beta_Q} & K \ar[d]^-{\epsilon_K} \ar[r]^-{g} & D \ar@{=}[d]\\
\Sigma^2 Q \ar[r]^-{\iota_f} \ar[d]_-{\beta_{\Sigma^2Q}}& C_f \ar[d]^-{\beta_{C_f}} \ar[r]^-{h} & D \ar[d]_-{\epsilon_D}\\
\Sigma^4Q \ar[r]^-{\Sigma^2\iota_f} & \Sigma^2 C_f \ar[r]^-{\iota_h} & C_h.}
\]
This shows that the pair $(f,g)$ can be lifted to a composable pair $(\Sigma^2\iota_f,\iota_h)$.
\end{proof}

\begin{rema}
The proof above depends on the fact that extensions with an equivariant, contractive, completely positive section can be shown to be isomorphic to mapping cone triangles. From an abstract standpoint, we may express this by saying that $\KKK^G$ is the result of the Verdier quotient \cite{hen:triloc,nee:tri} of the  Spanier-Whitehead category of $G$-$C^*$-algebras \cite{ivo:thesis} by the thick tensor ideal of objects $\Cone(\epsilon_K)$, for all canonical comparison maps $\epsilon_K$ associated to equivariant semi-split extensions (to be precise, we need to take into account yet another class of morphisms, to ensure that $\KKK^G$ is a stable functor, see \cite[Section A6.1]{ivo:thesis} and Definition \ref{def:eth} below).
\end{rema}

\begin{defi}\label{def:eth}
Let $\mathrm{SW}(C)$ be the Spanier-Whitehead category of the standard category of $G$-$C^*$-algebras, and let $\mathcal{I}\subseteq\mathrm{SW}(C)$ be the thick tensor ideal generated by the mapping cones of morphisms:
\begin{itemize}
    \item $\epsilon_K$ for any extension $K\hookrightarrow E \twoheadrightarrow Q$ in $C$;
    \item $\mathcal{K}(H_1)\to \mathcal{K}(H_1\oplus H_2)$ for any two nonzero $G$-Hilbert spaces $H_1,H_2$, where the map is induced by the canonical inclusion in the first factor. 
\end{itemize}
The equivariant $E$-theory category is defined as the Verdier quotient $E_G=\mathrm{SW}(C)/\mathcal{I}$.
\end{defi}

It should be clear from the definition above that $E_G$, viewed as functor from the category of separable $G$-$C^*$-algebras is the universal \emph{half-exact}, $C^*$-stable, and homotopy invariant functor. In this sense we can understand $E$-theory as the universal ``correction'' of $\KKK$-theory in terms of excision properties. The universal property implies in particular that any functor between ``concrete'' categories of $C^*$-algebras such as $f_*$ and $f^*$ extends to $E$-theory the same way it does for $\KKK$-theory.

By the same token, for a separable $G$-$C^*$-algebra $B$ we can define a functor $\sigma_B$ which is given by $\sigma_B(A)=A\otimes_X B$ on objects and $\sigma_B(\phi)=\phi\otimes \mathrm{1}_B$ on morphisms. It is important to discuss whether or not $\sigma_B$ is a triangulated functor on our $K$-theory categories $\KKK^G$ and $E_G$. By this we mean whether or not $\sigma_A$ preserves exact triangles. Since we are adopting the convention of using the maximal tensor product, the preservation of exact triangles is a simple consequence of the fact that $-\otimes B$ is an exact functor, and clearly it preserves semi-split extensions.

When $B$ is $C_0(X)$-nuclear, that is a continuous field over $X$ with nuclear fibers \cite{bauval:nuc}, we have an isomorphism $A\otimes_X B\cong (A\otimes^{\text{min}} B)_{\Delta_X}$ \cite{bla:defhopf}. Note that this applies in particular to the pullback functor $f^*$ associated to an open map $f\colon Y\to X$, such as the range and source maps $r,s\colon G\to G^{0}=X$. Thus if $B$ is exact or $C_0(X)$-nuclear the functor $\sigma_B$ is triangulated, regardless of the specific choice of tensor product.

The property of being $C_0(X)$-nuclear, or rather its $K$-theoretic counterpart called \emph{$\KKK^X$-nuclearity}, is important to establish a useful identification between $\KKK$- and $E$-theory groups as follows. More information on $\KKK^X$-nuclearity can be found in \cite{bauval:nuc}, here we limit ourselves to record the following simple fact, which is proved in \cite[Proposition 5.1 \& Corollary 5.2]{tu:moy} (see Definition \ref{def:proper} for proper groupoids).

\begin{prop}
Suppose $G$ is proper. If $A$ is $\KKK^{G^0}$-nuclear, for example $A$ is a continuous field over the unit space of $G$ with nuclear fibers, then the functor $B\mapsto \KKK^G(A,B)$ is half-exact.  
\end{prop}

Having this, the following is a simple consequence of the universal properties.

\begin{coro}[{\cite{patr:repEth}}]\label{coro:kkenuc}
If $G$ is proper and $A$ is a $\KKK^G$-nuclear $C^*$-algebra, there is a natural isomorphism $\KKK^G(A,B)\cong E_G(A,B)$ for any separable $G$-$C^*$-algebra $B$.
\end{coro}
\begin{proof}
Denote by $F$ the standard $\KKK$-functor from the category of separable $C^*$-algebras. The universal property of $\KKK$-theory gives us a map $\Phi_{C,B}\colon\KKK^G(C,B)\to E_G(C,B)$. Let $F^\prime$ be the functor (from separable $C^*$-algebras) given by $F^\prime(B)=\KKK^G(A,B)$ and $F^\prime(f\colon C\to B)$ induced by Kasparov product with $F(f)$. Since $\KKK^G(A,-)$ is half-exact, the universal property of $E$-theory yields a map $\Psi_{C,B}\colon \KKK^G(A,C)\times E_G(C,B)\to \KKK^G(A,B)$. It is clear that $\Psi(-\,,\Phi \circ F)=F^\prime$. In particular, for $f\colon A\to B$, we have
\[
\Psi_{A,B}(1_A, \Phi_{A,B}F(f))=F^\prime(f)(1_A)=F(f),
 \]
 which implies that $\Psi_{A,B}(1_A,-)$ is a left inverse for $\Phi_{A,B}$. The argument for showing it is a right inverse is analogous.
\end{proof}

\subsection{Complementary subcategories and cellular approximation}\label{subsec:compcell}

In this subsection we recall some facts about complementary subcategories, homotopy colimits in triangulated categories, and the fundamental notion of \emph{cellular approximation}. The material in this section is summarized from \cite{mey:catasp,meyer:tri,nestmeyer:loc,meyernest:tri}.

Let $F\colon \T \to \S$ be an exact functor between triangulated categories. This means that $F$ intertwines suspensions and preserves exact triangles. The kernel of $F$ (on morphisms), denoted $\I=\ker F$, will be called a \emph{homological ideal} (see \cite[Remark 19]{meyernest:tri}). We say that $\I$ is \emph{compatible with direct sums} if $F$ commutes with countable direct sums (see \cite[Proposition 3.14]{meyer:tri}). Note that triangulated categories involving $\KKK$-theory have no more than countable direct sums, because separability assumptions are needed for certain analytical results in the background.

An object $P\in \T$ is called \emph{$\I$-projective} if $\I(P,A)=0$ for all objects $A\in \T$. An object $N\in \T$ is called \emph{$\I$-contractible} if $\mathrm{id}_N$ belongs to $\I(N,N)$. Reference to $\I$ is often omitted in the sequel. Let $P_\I, N_\I \subseteq\T$ be the full subcategories of projective and contractible objects, respectively. 

We denote by $\langle {P_\I}\rangle$ the \emph{localizing} subcategory generated by the projective objects, i.e., the smallest triangulated subcategory that is closed under countable direct sums and contains $P_\I$. In particular, $\langle P_\I\rangle$ is closed under isomorphisms, suspensions, and if
\[
\xymatrix{A \ar[r] & B \ar[r] & C \ar[r] & \Sigma A}
\]
is an exact triangle in $\T$ where any two of the objects $A,B,C$ are in $\langle P_\I\rangle$, so is the third. Note that $N_\I$ is localizing, and any localizing subcategory is \emph{thick}, that is, closed under direct summands (see \cite{nee:tri}). 

\begin{defi}
Given an object $A\in \T$ and a chain complex
\begin{equation}\label{eq:projres}
\xymatrix{\cdots \ar[r]^-{\delta_{n+1}} & P_n \ar[r]^-{\delta_n} & \cdots \ar[r]^-{\delta_1} & P_0 \ar[r]^-{\delta_0} & A}
\end{equation} 
we say that \eqref{eq:projres} is a \emph{projective resolution} of $A$ if
\begin{itemize}
\item all the $P_n$'s are projective;
\item the chain complex below is split exact
\[\xymatrixcolsep{3.2pc}
\xymatrix{F(P_\bullet)\ar[r]^-{F(\delta_0)} & F(A)  \ar[r] & 0.}
\]
\end{itemize}
\end{defi}
We say that $\T$ has \emph{enough projectives} if any object admits a projective resolution.

\begin{prop}[{\cite[Proposition 44]{meyernest:tri}}]
The construction of projective resolutions yields a functor $\T \to \mathrm{Ho}(\T)$. In particular, two projective resolutions of the same object are chain homotopy equivalent.
\end{prop}

\begin{defi}
We call two thick triangulated subcategories $\P,\Nn$ of $\T$ \emph{complementary} if $\T(P,N)=0$ for all $P\in \P,N\in\Nn$ and, for any $A\in \T$, there is an exact triangle
\[
\xymatrix{ P \ar[r] & A \ar[r] & N \ar[r] & \Sigma P}
\]
where $P\in \P$ and $N\in \Nn$.
\end{defi}

\begin{prop}[{\cite[Proposition 2.9]{nestmeyer:loc}}]\label{prop:csub}
Let $(\P,\Nn)$ be a pair of complementary subcategories of $\T$.
\begin{itemize}
\item We have $N\in \Nn$ if and only if $\T(P,N)=0$ for all $P\in \P$. Analogously, we have $P\in \P$ if and only if $\T(P,N)=0$ for all $N\in \Nn$.
\item The exact triangle $P \to A \to N \to \Sigma P$ with $P\in \P$ and $N\in \Nn$ is uniquely determined up to isomorphism and depends functorially on $A$. In particular, its entries define functors
\begin{align*}
P\colon \T & \to \P & N\colon \T & \to \Nn\\
A &\mapsto P & A &\mapsto N.
\end{align*}
\item The functors $P$ and $N$ are respectively left adjoint to the embedding functor $\P\to \T$ and right adjoint to the embedding functor $\Nn \to \T$.
\item The localizations $\T/\Nn$ and $\T/\P$ exist and the compositions
\begin{align*}
\P &\longrightarrow \T \longrightarrow  \T/\Nn \\
\Nn &\longrightarrow \T \longrightarrow  \T/\P
\end{align*}
are equivalences of triangulated categories (see \cite{hen:triloc} for localization).
\item If $K\colon \T \to \mathcal{C}$ is a covariant functor, then its \emph{localization} with respect to $\Nn$ is defined by $\mathbb{L}K=K\circ P$ and the natural maps $P(A)\to A$ provide a natural transformation $\mathbb{L}K \Rightarrow K$.
\end{itemize}
\end{prop}

The following result will be very important for us.

\begin{theo}[{\cite[Theorem 3.16]{meyer:tri}}]
Let $\T$ be a triangulated category with countable direct sums, and let $\I$ be a homological ideal with enough projective objects. Suppose that $\I$ is compatible with countable direct sums. Then the pair of localizing subcategories $(\PI,N_\I)$ in $\T$ is complementary. 
\end{theo}

A pair of complementary subcategories helps clarify the degree to which a projective resolution ``computes'' a homological functor into the category of abelian groups. The object $P(A)$ resulting from Proposition \ref{prop:csub} is called the \emph{$P_\I$-cellular approximation} of $A$ (it is called \emph{simiplicial} approximation in \cite{nestmeyer:loc}). 

\begin{defi}
In general, the homotopy direct limit of a countable inductive system $(A_n,\alpha_m^n)$ is defined as the object $A^h_\infty$ fitting into the exact triangle below,
\[
\xymatrix{\bigoplus A_n \ar[r]^{\mathrm{id} - S} & \bigoplus A_n \ar[r] & A^h_\infty \ar[r] & \Sigma \bigoplus A_n .}
\]
where $S|_{A_n}\colon A_n\to A_{n+1}$ is just the connecting map $\alpha_n^{n+1}$. We write $\text{ho-lim}(A_n,\alpha_m^n)=A^h_\infty$, or simply $\hlim A_n$ when the connecting maps are clear from context.
\end{defi} 
\begin{rema}
The object $P(A)$ can be computed as the homotopy limit of an inductive system $(P_n,\phi_n)$ with $P_n \in P_\I$ (in fact, $P_n$ belongs to a subclass of objects in $P_\I$, see \cite[Proposition 3.18]{meyer:tri} for more details). 
\end{rema}
We mention a few more properties of this limit that will be useful for our later arguments. First of all, the last map in the triangle above is equivalent to a sequence of maps $\alpha_n^\infty\colon A_n \to A^h_\infty$ with the compatibility relation $\alpha_n^\infty \circ \alpha_m^n=\alpha_m^\infty$ when $m\leq n$. 

\begin{lemm}[{\cite{nee:tri}}]\label{lem:millimone}
Suppose $F$ is a (co)homological functor, i.e., it sends exact triangles to long exact sequences of abelian groups. 
\begin{itemize}
\item (homological case): if $F(\bigoplus A_n)\cong \bigoplus F(A_n)$, then the maps $\alpha_n^\infty$ give an isomorphism $\varinjlim F_k(A_n)\cong F_k(A^h_\infty).$
\item (cohomological case): if $F(\bigoplus A_n)\cong \prod F(A_n)$, there is a short exact sequence
\[
0\longrightarrow \varprojlim{}^1F^{k-1}(A_n) \longrightarrow F^k(A^h_\infty) \longrightarrow \varprojlim F^k(A_n)\longrightarrow 0,
\]
where the last map is induced by $(\alpha_n^\infty)_{n\in\N}$.
\end{itemize}
\end{lemm}

Let us consider the ordinary inductive limit of $C^*$-algebras $A_\infty$ associated to the system $(A_n,\alpha_m^n)$, where the maps $\alpha_m^n$ are equivariant $\ast$-homomorphisms. We keep using $\alpha_n^\infty$ for the canonical maps $A_n\to A_\infty$. The relation between $A^h_\infty$ and $A_\infty$, as discussed in \cite[Section 2.4]{nestmeyer:loc}, is based on the notion of an \emph{admissible} system in $\KKK^G$. We do not need this definition here, but we recall a sufficient condition: the system $(A_n,\alpha_m^n)$ is \emph{admissible} if there exist equivariant completely positive contractions $\phi_n\colon A_\infty \to A_n$ such that $\alpha_n^\infty\circ \phi_n\colon A_\infty\to A_\infty$ converges to the identity in the point norm topology \cite[Lemma 2.7]{nestmeyer:loc}. The situation is simpler in $E_G$-theory: by Definition \ref{def:eth}, since all extensions in $E_G$-theory are admissible, all inductive systems are admissible too. 

\begin{prop}\label{prop:admlim}
We have $A^h_\infty\cong A_\infty$ in the category $E_G$. If the inductive system $(A_n,\alpha_m^n)$ is admissible, we have $A^h_\infty\cong A_\infty$ in the category $\KKK^G$.
\end{prop}

\subsection{Crossed products of Hilbert modules and descent}

In this section we recall the notion of crossed product of Hilbert modules and define the Kasparov descent morphism in the context of groupoids. We will focus on \emph{reduced} crossed products. To this end, we start by recasting $C_0(X)$-algebras under the perspective of $C^*$-bundles. If $A$ is a $C_0(X)$-algebra, there exists a topology on $\mathcal{A}=\bigsqcup_{x\in X} A_x$ making the natural map $\mathcal{A}\to X$ an upper-semicontinuous $C^*$-bundle. The associated algebra of sections vanishing at infinity, denoted $\Gamma_0(X,\mathcal{A})$, admits a $C_0(X)$-linear isomorphism onto $A$. The correspondence $A\mapsto \mathcal{A}$ sends $C_0(X)$-linear morphisms to $C^*$-bundles morphisms.

If $f\colon Y \to X$ is a continuous map, the pullback $C^*$-algebra $f^*A$ can also be defined by first constructing the pullback bundle $f^*\mathcal{A}$, then setting $f^*A=\Gamma_0(Y,f^*\mathcal{A})$. A $G$-action on $A$ can be given by defining a functor from $G$ (viewed as a category) to the category of $C^*$-algebras, sending $x\in X$ to $A_x$, then imposing continuity on the resulting $G$-action on the topological space $\mathcal{A}$. The definition of $A\rtimes G$ can then be reframed by endowing the compactly supported sections $\Gamma_c(G,r^*\mathcal{A})$ with a $\ast$-algebra structure, and completing in the appropriate norm as explained previously.

Given a $G$-algebra $(A,\alpha)$ and a Hilbert $A$-module $\EE$, for each $x\in X$ one defines the Hilbert $A_{x}$-module $\EF_{x}$ to be the balanced tensor product $\EE\otimes_{A}A_{x}$.  The space $\EF:=\bigsqcup_{x\in X}\EF_{x}$ may be topologized to obtain an upper-semicontinuous Hilbert $\AF$-module bundle $p_{\EF}:\EF\longrightarrow X$. The space of sections $\Gamma_{0}(X;\EF)$ is equipped with pointwise operations to furnish a Hilbert $\Gamma_{0}(X;\AF)$-module, to which $\EE$ is canonically isomorphic as a Hilbert $A$-module. We will identify $\EE$ with its associated section space $\Gamma_{0}(X;\EF)$.
We have associated bundles of $C^{*}$-algebras $\K(\EF)$ and $\LL(\EF)$, whose fibres over $x\in X$ are $\K(\EF_{x})$ and $\LL(\EF_{x})$, respectively (the former bundle is upper-semicontinuous).  By the identification $\EE = \Gamma_{0}(X;\EF)$, we then also have $\K(\EE) = \Gamma_{0}(X;\K(\EF))$ and $\LL(\EE) = \Gamma_{b}(X;\LL(\EF))$ (strictly continuous bounded sections).

A $G$-action $\EE = \Gamma_{0}(X;\EF)$ consists of a family $\{W_{\gamma}\}_{\gamma\in \GG}$ such that:
	\begin{itemize}
		\item for each $\gamma\in \GG$, $W_{\gamma}:\EF_{s(\gamma)}\longrightarrow \EF_{r(\gamma)}$ is an isometric isomorphism of Banach spaces such that $\langle W_{\gamma}e, W_{\gamma}f\rangle_{r(\gamma)} = \alpha_{\gamma}(\langle e,f\rangle_{s(\gamma)})$ for all $e,f\in\EF_{s(\gamma)}$;
		\item the map $\GG\tensor[_s]{\times}{_{p_{\EF}}}\EF\longrightarrow\EF$, $(\gamma,e)\mapsto W_{\gamma}e$ defines a continuous action of $\GG$ on $\EF$.
	\end{itemize}
Conjugation by $W$ gives rise to a strictly continuous action $\varepsilon:\GG\tensor[_s]{\times}{_{p_{\EF}}}\LL(\EF)\longrightarrow\LL(\EF)$ of $\GG$ on the upper semicontinuous bundle $\LL(\EF)$ (the restriction of $\varepsilon$ to the compact operators is continuous in the usual sense).

If $(B, \beta) $ is a $G$-algebra and $\pi\colon B\to \LL(\EE)$ a $C_0(X)$-linear representation, we define a $G$-representation by requiring equivariance, namely for all $\gamma\in\GG$ we have
	\[
	\varepsilon_{\gamma}\circ\pi_{s(\gamma)} = \pi_{r(\gamma)}\circ \beta_{\gamma}.
	\]
Given a Kasparov module $(\pi,\EE,T)$ representing a class in $\KKK^G(B,A)$, let us consider the $B\rtimes_r G$-$A\rtimes_r G$-module $(\tilde{\pi},\EE\hot_A (A\rtimes_r G), T\hot 1)$ where $\tilde{\pi}$ is a representation of $B\rtimes_r G$ induced by $\pi$ as follows. First of all, note that $\EE\hot_A (A\rtimes_r G)$ is isomorphic to the completion of $\Gamma_c(G,r^*\EF)$ with respect to the $\Gamma_{c}(\GG;r^{*}\AF)$-valued inner product
\[
\langle\xi,\xi'\rangle(\gamma):=\int_{\GG}\alpha_{\eta}\big(\langle\xi(\eta^{-1}),\xi'(\eta^{-1}\gamma)\rangle_{s(\eta)}\big)\,d\lambda^{r(\gamma)}(\eta),
\]
for $\xi,\xi'\in\Gamma_{c}(\GG;r^{*}\EF)$ and $\gamma\in G$. We denote this completion $\EE\rtimes G$. Consider the formula below, defined for $f\in\Gamma_{c}(\GG;r^{*}\AF)$, $\xi\in\Gamma_{c}(\GG;r^{*}\EF)$, and $\gamma\in\GG$,
	\[
	(f\cdot\xi)(\gamma):=\int_{\GG}\pi_{r(\eta)}(f(\eta))W_{\eta}\big(\xi(\eta^{-1}\gamma)\big)\,d\lambda^{r(\gamma)}(\eta).
	\]
This determines a bounded representation $\tilde{\pi}=\pi\rtimes G :A\rtimes_r\GG\longrightarrow\LL(\EE\rtimes G)$ (see for example \cite[Prop.~7.6]{macd:kk}).

\begin{defi}
We define the Kasparov \emph{descent} morphism to be the homomorphism of abelian groups
\[
\jmath^G\colon \KKK^G(B,A)\to \KKK(B\rtimes_r G, A\rtimes_r G)
\]
which sends the class of $(\pi,\EE, T)$ to the class of $(\tilde{\pi},\EE\hot_A (A\rtimes_r G),T\hot 1)$.
\end{defi}

It can be checked that $\jmath^G$ is compatible with the product in $\KKK^G$, meaning that $\jmath^G(x\,\hot_D\, y)=\jmath^G(x)\,\hot_{D\rtimes_r G}\, \jmath^G(y)$, giving us a well-defined functor \cite[Theorem~3.4]{leGall:Survey}.

\section{Induction-restriction adjunction}
\label{sec:indres}

Consider a subgroupoid $H\subseteq G$. The inclusion map $H\hookrightarrow G$ induces a natural restriction functor $\Res_G^H\colon \KKK^G\to \KKK^H$. In this section we will construct a functor in the other direction, called the induction functor, and prove that these two functors are adjoint when $H\subseteq G$ is \emph{open}. This generalizes earlier results for transformation groups \cite{nestmeyer:loc} and ample groupoids \cite{bon:goingdownbc}.

\subsection{The induction functor} 

Let $(B,\beta)\in \KKK^H$ with moment map $\rho\colon C_0(H^{0})\to Z(\mathcal{M}(B))$. In this subsection it is sufficient to assume $H$ is locally closed in $G$. Recall $G_{H^{0}}$ is the subspace of $G$ consisting of arrows with source in $H^{0}$.
We consider the restriction of the source map $\phi= s|_{ G_{H^0}}:G_{H^{0}}\rightarrow H^{0}$, and construct the pullback algebra
\[
\phi^*B= C_0(G_{H^{0}}) \tens{s}{H^{0}}{\rho} B
\]
This balanced tensor product is then a $C_0(H^{0})$-algebra in its own right and can be equipped with the diagonal action $\text{rt}\otimes \beta$ of $H$, where $\text{rt}$ denotes the action of $H$ on $C_0(G_{H^{0}})$ induced by right translation. We define the induced algebra as the corresponding reduced crossed product
\[
\Ind_H^G B:=(C_0(G_{H^{0}}) \tens{s}{H^{0}}{\rho} B)\rtimes_{\text{rt}\otimes \beta} H.
\]
To define a $G$-action on $\Ind_H^G B$, notice that $G$ also acts on the balanced tensor product $C_0(G_{H^{0}}) \otimes_{H^{0}} B$ by $\mathrm{lt}\otimes \mathrm{id}_B$, where $\mathrm{lt}$ denotes the action of $G$ on $C_0(G_{H^{0}})$ induced by left translation. A straightforward computation reveals that the actions $\mathrm{rt}\otimes\beta$ and $\mathrm{lt}\otimes \mathrm{id}_B$ commute and therefore the left translation action of $G$ descends to an action on the crossed product $(C_0(G_{H^{0}}) \otimes_{H^{0}}B)\rtimes_{\text{rt}\otimes \beta} H.$

Having defined $\Ind_H^G$ on objects, let us consider the case of morphisms. Consider a right Hilbert $B$-module $\EE$. Considering the canonical action $B\longrightarrow \mathcal{M}(C_0(G_{H^{0}})\otimes_{H^{0}} B)$ given by multiplication in the second factor, we can form the $\phi^*B$-module
\[
\phi^*\EE=\EE\otimes_B \left(C_0(G_{H^{0}})\otimes_{H^{0}} B\right).
\]
Note the module above corresponds to the space of section of the pullback bundle $\phi^*\EF$.
Assume now that $\EE$ carries an action of $H$ (call it $\epsilon$) along with a non-degenerate equivariant representation $\pi\colon A\to \mathcal{L}(\EE)$ of an $H$-algebra $A$. First of all, we note that $\epsilon\otimes (\mathrm{rt}\otimes \beta)$ defines an $H$-action on $\phi^*\EE$. Then we define a representation of $\phi^*A$ on  $\mathcal{\phi^*\EE}$ by considering elements $f\otimes a$, with $f\in C_c(G_{H^{0}})$ and $a\in A$, whose linear span is dense in $\Gamma_c(G_{H^{0}},\phi^*\AF)\subseteq \phi^*A$, and setting $\phi^*\pi(f\otimes a) = \pi(a)\otimes (f\,\cdot)$.

Now, if $(\pi,\EE,T)$ is an $A$-$B$-Kasparov module, then $(\phi^*{\pi},\phi^*\EE,T\hot 1)$ is a $\phi^*A$-$\phi^*B$-module equipped with an action of $H$, and we can define the induction functor by means of the descent morphism defined above, as follows:
\[
\Ind_H^G({\pi},\EE,T)= \jmath_H(\phi^*{\pi},\phi^*\EE,T\hot 1).
\]

To complete the description of $\Ind_H^G$, we need two more observations. The $\phi^*B$-module $\EE\otimes_B (C_0(G_{H^{0}})\otimes_{H^{0}} B)$ admits a $G$-action induced by left translation on $C_0(G_{H^{0}})$. Notice this action is defined by fibering over the range map. Clearly $T\hot 1$ is equivariant with respect to this translation. To check the equivariance of $\phi^*\pi$, by definition it is sufficient to consider $\gamma\in G$ and $f\in C_c(G_{H^{0}})$, and write
\[
[\gamma\cdot(f\cdot(\gamma^{-1}\cdot g))](\eta) = f(\gamma\eta)g(\gamma\gamma^{-1}\eta)= (\mathrm{lt}_\gamma(f)\cdot g)(\eta)
\]
with $g\in C_0(G_{H^{0}})$, $\eta\in G_{H^{0}}$ with $r(\eta)=s(\gamma)$. This ensures the $G$-action commutes with the $H$-action on $(\phi^*{\pi},\phi^*\EE,T\hot 1)$, hence $\jmath_H(\phi^*{\pi},\phi^*\EE,T\hot 1)\in\KKK^G(A,B)$. Finally, as $\Ind_H^G$ is defined as a composition of the pullback functor $\phi^*$ with the descent functor $\jmath^G$, it is indeed a functor $\Ind_H^G\colon\KKK^H\to\KKK^G$.

\begin{rema}\label{rem:indmodel}
Both the descent functor $\jmath_G:\KKK^G\rightarrow \KKK$ and the induction functor $\Ind_H^G:\KKK^H\rightarrow \KKK^G$ can be abstractly constructed using the universal property of equivariant $\KKK$-theory, by observing that the respective constructions on the $C^*$-level are compatible with split-exact sequences, stabilisations, and homotopies (compare \cite{meyernest:tri}). In many applications however it is useful to have a concrete model at hand. This is certainly the case for the adjunction result in Theorem \ref{thm:iradj} below, but has also proven to be a useful construction in \cite{valmak:groupoid, bdgw:matui}.

The model for the induction functor in \cite{bon:goingdownbc} is different from the one employed here. Given an $H$-$C^*$-algebra $A$, the construction of $\Ind_H^G(A)$ in \cite{bon:goingdownbc} prescribes constructing the pullback algebra $\phi^*A= C_0(G_{H^0})\tensor[_s]{\otimes}{_\rho} A$ as above, but then considers the (generalized) fixed-point algebra $\phi^*A^H$ associated to the diagonal $H$-action. If $H$ is acting properly on $G$, then the main result in \cite{brown:proper} implies that $\phi^*A^H$ is strongly Morita equivalent to $\Ind_H^G(A)$. It is not hard to see that the imprimitivity bimodule witnessing this equivalence gives a $G$-equivariant $\KKK$-equivalence. 

It should be noted that, when $H\subseteq G$ is closed (hence $G\rtimes H$ is proper), then the spectrum of $\phi^*C_0(Z)^H$ is homeomorphic to the ordinary induction space $G\times_H Z$ (see \cite[Proposition~3.22]{bon:goingdownbc}). However, if $H\subseteq G$ is open, then it need not act properly on $G$, and it is well-known that quotients by non-proper actions can lead to pathological topological spaces (e.g., non-Hausdorff, non-locally-compact). It is for this reason that in this paper, where induction from \emph{open} subgroupoids is considered, we have taken the approach of defining induction via crossed products.
\end{rema}

\subsection{Proof of the adjunction}

Recall that if $G$ acts freely and properly on a second countable, locally compact, Hausdorff space $Y$, then $G\ltimes Y$ is Morita equivalent as a groupoid to $Y/G$ and hence the groupoid $C^*$-algebra $C_0(Y)\rtimes G\cong C^*(G\ltimes Y)$ is strongly Morita equivalent to $C_0(Y/G)$ \cite{brown:proper}. Note that $G\ltimes Y$ is an amenable groupoid, so the reduced and full crossed products are isomorphic, see for example \cite[Corollary 2.1.17 \& Proposition 6.1.10]{renroch:amgrp}). 

In particular, when $Y$ equals $G$ itself and the action is given by right translation, the associated imprimitivity bimodule $X^G$ gives a $*$-isomorphism $C_0(G)\rtimes_{\text{rt}} G\cong \cK(L^2(G))$, where $L^2(G)$ is the standard continuous field of Hilbert spaces associated to $G$. The $\KKK$-class induced by $X^G$ will be important in a moment.

If $(A,G,\alpha)$ is a groupoid dynamical system, then the pushforward along the source map $s_*\alpha$ is an isomorphism of $C^*$-dynamical systems:
\[
s_*\alpha: (s_*(C_0(G)\tens{s}{G^{0}}{\rho}A),G,\text{rt}\otimes\alpha)\to (s_*(C_0(G)\tens{r}{G^{0}}{\rho}A),G,\text{rt}\otimes \mathrm{id}_A),
\]
where the intertwining map is given precisely by $\alpha$ \cite{gall:kk}.
As a consequence we have the following.
\begin{lemm} \label{lem:indresisom}
If $H\subseteq G$ is a locally closed subgroupoid and $A$ is a $G$-algebra, then we have a canonical isomorphism
\[
\Phi \colon \Ind_H^G\Res_G^H A \cong (C_0(G_{H^{0}})\rtimes_{\mathrm{rt}} H) \otimes_{G^{0}} A.
\]
After $\Phi$, the $G$-action on the right-hand side is given by $\text{lt}\otimes \alpha$, i.e., left translation on $C_0(G_{H^{0}})\rtimes_{\mathrm{rt}} H$, tensorized with the original action $\alpha$ on $A$.
\end{lemm}
\begin{proof}
Let $\alpha:s^*A\longrightarrow r^*A$ denote the $C_0(G)$-linear isomorphism implementing the action of $G$ on $A$. Now we can consider the pushforward along the source maps to obtain a $C_0(G^{0})$-linear isomorphism
$\alpha=s_*\alpha:s_*s^*A\longrightarrow s_*r^*A$. Now $s_*s^*A$ is just the balanced tensor product $C_0(G)\otimes_{G^{0}}A$ with the canonical $C_0(G^{0})$-algebra structure, while $s_*r^*A= \Gamma_0(G,r^*\mathcal{A})$ is equipped with the $C_0(G^{0})$-algebra structure obtained by the formula $(\varphi\cdot f)(g)=\varphi(s(g))f(g)$ for $\varphi\in C_0(G^{0})$ and $f\in \Gamma_0(G,r^*\mathcal{A})$. Note that this differs from the canonical structure it obtains as a balanced tensor product! With the structure defined above we can identify the fibre over a point $x\in G^{0}$ as $\Gamma_0(G,r^*\mathcal{A})_x=\Gamma_0(G_x,r^*\mathcal{A})$ and it makes sense to consider the action $\mathrm{rt}\otimes\mathrm{id}_A$ defined by
$$(\mathrm{rt}\otimes\mathrm{id}_A)_g(f)(h)=f(hg).$$

Summing up the discussion we see that $\alpha$ implements an isomorphism of groupoid dynamical systems 
\[
(C_0(G)\tens{s}{G^{0}}{\rho}A,G,\text{rt}\otimes\alpha)\to (C_0(G)\tens{r}{G^{0}}{\rho}A,G,\text{rt}\otimes \mathrm{id}_A).
\]
Now if we restrict these systems to the subgroupoid $H$ we obtain an isomorphism
\[
(C_0(G_{H^{0}})\tens{s}{H^{0}}{\rho}\Res^H_G A,H,\text{rt}\otimes\alpha)\to (C_0(G_{H^{0}})\tens{r}{G^{0}}{\rho} A,H,\text{rt}\otimes \mathrm{id}_A)
\]
In particular we obtain an isomorphism between the crossed products and hence conclude
\begin{align*}
    \Ind_H^G \Res_G^H A & = (C_0(G_{H^{0}}) \tens{s}{H^{0}}{\rho} \Res_G^H A)\rtimes_{r,\text{rt}\otimes \alpha} H\\
    &\cong (C_0(G_{H^{0}}) \tens{r}{G^{0}}{\rho}  A)\rtimes_{r,\text{rt}\otimes \mathrm{id}_A} H\\
    &\cong (C_0(G_{H^{0}})\rtimes_{\mathrm{rt}}H)\otimes_{G^{0}} A.
\end{align*}
\end{proof}

Choosing $H=G$ in the result above yields an isomorphism 
\[
\Ind_G^G\Res_G^G(B)\cong (C_0(G)\rtimes_{\text{rt}} G) \otimes_{G^{0}} B\cong \cK(L^2(G))\otimes_{G^{0}} B.
\]

We now prepare to prove the adjunction by defining some auxiliary maps. From now on we assume $H\subseteq G$ to be an open subgroupoid. We get an induced embedding 
\[C_0(G_{H^{0}})\rtimes_{\mathrm{rt}}H\hookrightarrow C_0(G)\rtimes_{\mathrm{rt}}G\]
and hence, using the previous Lemma, an embedding
\[
\kappa\colon \Ind_H^G\Res_G^H(B)\longrightarrow \mathcal{K}(L^2(G))\otimes_{G^{0}} B.
\]
We can promote $X^G$ to a $\KKK^G$-equivalence 
\[
X^G_A\in \KKK^G(\Ind_G^G\Res_G^G(A),A)
\]
given by the right $A$-module $L^2(G)\tensor[_r]{\otimes}{_\rho} A$, where $A$ acts pointwise as ``constant functions''. The representation of the crossed product $r^*A\rtimes G\cong \Ind_G^G\Res_G^G(A)$ is the integrated form of the covariant pair given by the right regular representation of $G$, and pointwise multiplication of functions in $r^*A$. We will denote this by $M_A\rtimes R_G$.

Now let $B\in \KKK^H$ and recall that 
\[
\Res_G^H\Ind_H^G(B)=(C_0(G|_{ H^{0}})\otimes_{H^{0}}B)\rtimes H
.\] Then the inclusion $C_0(H)\subseteq C_0(G|_{ H^{0}})$ induces a map
\[
\iota\colon \Ind_H^H B\cong (C_0(H)\otimes_{H^{0}}B)\rtimes H \to (C_0(G|_{ H^{0}})\otimes_{H^{0}}B)\rtimes H=\Res_G^H\Ind_H^G(B).
\]

\begin{theo}\label{thm:iradj} Let $G$ be a locally compact Hausdorff groupoid with Haar system. For every open subgroupoid $H\subseteq G$ there is an adjunction
\[
(\epsilon,\eta)\colon \Ind_H^G \dashv \Res_G^H
\]
with counit and unit
\begin{align*}
\epsilon\colon \Ind_H^G\Res_G^H&\to 1_{\KKK^G}\\
\eta\colon 1_{\KKK^H} &\to \Res_G^H\Ind_H^G
\end{align*}
described as follows:
\begin{align*}
\epsilon_A &= X^G_A\circ \kappa \\
\eta_B &=\iota \circ \bigl(X^{H}_B\bigr)^{\text{op}}.
\end{align*}
\end{theo}

Here below we isolate a couple of technical lemmas which will be useful in the proof of the adjunction. The first lemma is just an observation on the compatibility of the canonical element $X_A^G$ with restriction and induction.

\begin{lemm}\label{Lem:Restriction of canonical module}
Let $H\subseteq G$ be an open subgroupoid, and $A\in\KKK^G$. Then we have $X^H_{\Res_G^H A}=\sigma_{\Res^H_G A}(X_{C_0(H^{0})}^H)$ and $\Res_G^H(X_A^G)=\sigma_{\Res_G^H A}(\Res_G^H(X_{C_0(G^{0})}^G))$.
\end{lemm}
\begin{proof}
The first equality is immediate from the definition of $X_A^G$ and the isomorphism $\Ind_G^G\Res_G^G(A)\cong \cK(L^2(G))\otimes_{G^{0}} A$ explained above. The second equality follows from the first and the fact that restriction and tensorization commute.
\end{proof}

Let $L^2(G,B)$ denote the completion of $\Gamma_c(G,r^*\mathcal{B})$ with respect to the $B$-valued inner product $\langle \xi_1,\xi_2\rangle(x)=\int_{G^x} \xi_1(g)^*\xi_2(g)\,d\lambda^x(g)$. Note that $L^2(G,B)$ is canonically isomorphic to the $B$-module $L^2(G)\otimes_{G^{0}} B$ introduced above.

Let us make a point on notation before continuing the proof. So far we have used $A$ and $\AF$ to denote a $C_0(X)$-$C^*$-algebra and its corresponding $C^*$-bundle. However, this difference in font is not very convenient when $A$ is replaced by a more complicated algebra, e.g., $A=C_0(G)\rtimes H$. In the sequel we suppress this notational distinction, as the context suffices to disambiguate the usage.

\begin{lemm}\label{Lem:Ind of canonical module} Let $H\subseteq G$ be an open subgroupoid and $B\in\KKK^H$. Then there is an isometric $G$-equivariant homomorphism
\[
\Phi\colon\Ind_H^G L^2(H,B) \longrightarrow L^2(G,\Ind_H^G B)
\]of Hilbert $\Ind_H^G B$-modules.

\end{lemm}
\begin{proof}
Let us first describe the module $\Ind_H^G L^2(H,B)$ more concretely. We have a canonical isomorphism $L^2(H,B)\otimes_B (C_0(G_{H^{0}})\otimes_{H^{0}} B)\cong L^2(H,C_0(G_{H^{0}})\otimes_{H^{0}} B)$ given by $\xi\otimes f\mapsto [h\mapsto \xi(h)f]$. Hence we can write $\Ind_H^G L^2(H,B)$ as $L^2(H,C_0(G_{H^{0}})\otimes_{H^{0}} B)\rtimes H$. So for a function $\xi\in \Gamma_c(H,r^*L^2(H,C_0(G_{H^{0}})\otimes_{H^{0}} B))$ we define $\Phi(\xi)\in L^2(G,\Ind_H^G B)$ as
$$\Phi(\xi)(g,h,x)=\left\{\begin{array}{ll}
  \beta_{x^{-1}g}(\xi(g^{-1}xh,g^{-1}x,g)),   & g^{-1}x\in H  \\
    0 ,& \text{otherwise}
\end{array}\right\}$$
where $g\in G$, $h\in H$, and $x\in G_{r(h)}^{r(g)}$.

Given $\xi_1,\xi_2 \in \Gamma_c(H,r^*L^2(H,C_0(G_{H^{0}})\otimes_{H^{0}} B))$, we compute (for $h\in H$ and $x\in G_{r(h)}$) that $\langle \Phi(\xi_1),\Phi(\xi_2)\rangle (h,x)$ equals
\begin{align*}
    &\int\limits_{G}\left[\Phi(\xi_1)(g)^\ast\Phi(\xi_2)(g)\right](h,x)\,d\lambda^{r(x)}(g)\\
    = &\int\limits_{G} \int\limits_{H}(\mathrm{rt}\otimes\beta)_{\tilde{h}}(\Phi(\xi_1)(g,\tilde{h}^{-1})^*\Phi(\xi_2)(g,\tilde{h}^{-1}h))(x)\,d\lambda^{r(h)}(\tilde{h})\,d\lambda^{r(x)}(g)\\
    = &\int\limits_{G} \int\limits_{H}\beta_{\tilde{h}}(\Phi(\xi_1)(g,\tilde{h}^{-1},x\tilde{h})^*\Phi(\xi_2)(g,\tilde{h}^{-1}h,x\tilde{h}))\,d\lambda^{r(h)}(\tilde{h})\,d\lambda^{r(x)}(g)\\
    =& \int\limits_{xH} \int\limits_{H}\beta_{x^{-1}g}(\xi_1(g^{-1}x,g^{-1}x\tilde{h},g)^*\xi_2(g^{-1}xh,g^{-1}x\tilde{h},g))\,d\lambda^{r(h)}(\tilde{h})\,d\lambda^{r(x)}(g).
\end{align*}

At this point we perform two change of variables and keep computing:
\begin{align*}
    &\stackrel{g\mapsto xg}{=}\int\limits_{H}\int\limits_{H}\beta_g(\xi_1(g^{-1},g^{-1}\tilde{h},xg)^*\xi_2(g^{-1}h,g^{-1}\tilde{h},xg))\,d\lambda^{r(h)}(\tilde{h})\,d\lambda^{s(x)}(g)\\
    & \stackrel{\tilde{h}\mapsto g\tilde{h}}{=}\int\limits_{H}\int\limits_{H}\beta_g(\xi_1(g^{-1},\tilde{h},xg)^*\xi_2(g^{-1}h,\tilde{h},xg))\,d\lambda^{s(g)}(\tilde{h})\,d\lambda^{s(x)}(g)\\
    & = \int\limits_{H}\int\limits_{H}\beta_{g^{-1}}(\xi_1(g,\tilde{h},xg^{-1})^*\xi_2(gh,\tilde{h},xg^{-1}))\,d\lambda^{r(g)}(\tilde{h})\,d\lambda_{s(x)}(g)\\
    & = \int\limits_{H} (\mathrm{rt}\otimes \beta)_{g^{-1}}(\langle \xi_1(g),\xi_2(gh)\rangle(x)\,d\lambda_{r(h)}(g)\\
    & = \langle \xi_1,\xi_2\rangle (h,x)
\end{align*}
This verifies that $\Phi$ extends to an isometry.
Now we proceed to checking that $\Phi$ is a right module map. Below we have $\xi \in \Gamma_c(H,r^*L^2(H,C_0(G_{H^{0}})\otimes_{H^{0}} B))$ as before, and the element $f$ belongs to $\Gamma_c(H,r^*(C_0(G_{H^{0}})\tensor[_s]{\otimes}{_\rho}B))$.

\begin{align*}
    (\Phi(\xi)f)(g,h,x) &= \Phi(\xi)(g,h,x)f(h,x)\\
    &=\int_{H^{r(g)}} \Phi(\xi)(g,\tilde{h},x)\beta_{\tilde{h}}(f(\tilde{h}^{-1}h,x\tilde{h}))\,d\lambda^{r(h)}(\tilde{h})\\
    & = \int_{H^{r(g)}} \beta_{x^{-1}g}(\xi(g^{-1}x\tilde{h},g^{-1}x,g))\beta_{\tilde{h}}(f(\tilde{h}^{-1}h,x\tilde{h}))\,d\lambda^{r(h)}(\tilde{h})\\
    & \stackrel{\tilde{h}\mapsto x^{-1}g\tilde{h}}{=} \int_{H^{s(g)}} \beta_{x^{-1}g}(\xi(\tilde{h},g^{-1}x,g)\beta_{\tilde{h}}(f(\tilde{h}^{-1}g^{-1}xh,g\tilde{h})))\,d\lambda^{s(g)}(\tilde{h})\\
    & = \beta_{x^{-1}g}((\xi f)(g^{-1}xh,g^{-1}x,g))\\
    & = \Phi(\xi f)(g,h,x)
\end{align*}

To complete the argument, we show that the left action of $G$ commutes with $\Phi$. Let us take $g'\in G$ with $r(g')=r(g)$ and compute
\begin{align*}
(g'\Phi(\xi))(g,h,x)&=\Phi(\xi)(g'^{-1}g,h,g'^{-1}x)\\
&=\beta_{x^{-1}g} (\xi(g^{-1}xh,g^{-1}x,g'^{-1}g))\\
&=\beta_{x^{-1}g}((g' \xi)(g^{-1}xh,g^{-1}x,g))\\
&=\Phi(g'\xi)(g,h,x)
\end{align*}
The proof is complete.
\end{proof}

\begin{proof}[Proof of Theorem \ref{thm:iradj}] 
We need to verify the counit-unit equations. We start by proving that for every $A\in \KKK^G$ the composition 
\[
\xymatrixcolsep{5pc}\xymatrix{\Res_G^H A \ar[r]^-{\eta_{\Res_G^H A}}& \Res_G^H\Ind_H^G\Res_G^H A \ar[r]^-{\Res_G^H(\epsilon_A)} & \Res_G^H A}
\]
equals the identity in $\KKK^H(\Res_G^HA,\Res_G^HA)$:
Expanding the definitions of counit and unit in this case we have
$\Res_G^H(\epsilon_A)\circ\eta_{\Res_G^H A} = \Res (X_A^G)\circ \Res(\kappa)\circ \iota\circ (X_{\Res_G^H A}^H)^{\text{op}}$.
Following the definitions it is then easily seen that after identifying
\begin{gather*}
\Ind_H^H(\Res_G^H A)=(C_0(H)\rtimes_{\mathrm{rt}}H)\otimes \Res_G^HA\\
\Res_G^H(\Ind_G^G A)\cong \Res_G^H((C_0(G)\rtimes_{\mathrm{rt}}G)\otimes_{G^{0}} A)\cong (C_0(G^{H^{0}})\rtimes_{\mathrm{rt}}G)\otimes_{H^{0}}\Res_G^H A
\end{gather*}
the composition
$\Res(\kappa)\circ\iota$ is just given by
\begin{equation}\label{Eq:kappaiota}
(C_0(H)\rtimes H)\otimes_{H^{0}} \Res_G^H A\stackrel{j\otimes \mathrm{id}}{\longrightarrow}(C_0(G^{H^{0}})\rtimes G)\otimes_{H^{0}} \Res_G^H A,
\end{equation}
where $j:C_0(H)\rtimes H\longrightarrow C_0(G^{H^{0}})\rtimes G$ is induced by the inclusion of $H$ as an open subgroupoid. Using Lemma \ref{Lem:Restriction of canonical module} we have
\begin{align*}
    \Res_G^H(\epsilon_A)\circ\eta_{\Res_G^H A} &=
    \Res (X_A^G)\circ \Res(\kappa)\circ \iota\circ (X_{\Res_G^H A}^H)^{\text{op}}\\
    &=\sigma_{\Res_G^H A}(\Res_G^H(X_{C_0(G^{0})}^G) \circ \sigma_{\Res_G^H A}(j)\circ \sigma_{\Res_H^G A}((X_{C_0(H^{0})}^H)^{\text{op}})\\
    &= \sigma_{\Res_G^H A}\left(\Res_G^H(X_{C_0(G^{0})}^G)\circ j\circ (X_{C_0(H^{0})}^H)^{\text{op}}\right) \end{align*}
Hence it is enough to show that the conclusion holds for $A=C_0(G^{0})$.
In this case we can further use the isomorphisms $C_0(H)\rtimes_{\mathrm{rt}} H\cong \mathcal{K}(L^2(H))$ and $C_0(G^{H^{0}})\rtimes_{\mathrm{rt}} G\cong \mathcal{K}(L^2(G^{H^{0}}))$ to replace the map in (\ref{Eq:kappaiota}) by the canonical map
\[
i\colon \cK(L^2(H))\to \cK(L^2(G^{H^{0}}))
\]
and the required verification is easily seen to be reduced to showing that the (interior) Kasparov product
\[
[(X^H_{C_0(H^{0})})^{\text{op}}]\hatotimes_{\cK(L^2(H))} i^*[\Res_G^H(X^G_{C_0(G^{0})})]
\]
equals the class of identity $id_{C_0(H^0)}$ in $\KKK^H(C_0(H^{0}),C_0(H^{0}))$. 

The element $\Res_G^H(X_{C_0(G^{0})}^G)\in \KKK^H(\mathcal{K}(L^2(G^{H^{0}})),C_0(H^{0}))$ can be represented by the triple $(L^2(G^{H^{0}}),\Phi,0)$, where $\Phi$ is the canonical action. Consequently, $i^*[\Res_G^H(X^G_{C_0(G^{0})})$ is represented by $(L^2(G^{H^{0}}),\Phi\circ i,0)$.
The representation $\Phi\circ i$ fails to be non-degenerate, but we can replace $L^2(G^{H^{0}})$ by its ``non-degenerate closure'' $\overline{\Phi\circ i(\mathcal{K}(L_s^2(H)))L^2(G^{H^{0}})}$ without changing its $\KKK^H$-class (see \cite[Proposition 18.3.6]{black:kth}). This module is easily seen to be (isomorphic to) $L^2(H)$. Therefore $i^*[\Res_G^H(X^G_{C_0(G^{0})})]=[X^H_{C_0(H^{0})}]$ and the desired equality follows from
\[
[(X^H_{C_0(H^{0}))})^{\text{op}}]\hatotimes_{\cK(L^2(H))} [X^H_{C_0(H^{0}))}]=1\in \KKK^H(C_0(H^{0}),C_0(H^{0})).
\]
The next verification in order regards the composition
\begin{equation}\label{eq:unit-counit}
\xymatrixcolsep{5pc}\xymatrix{\Ind_H^G(A) \ar[r]^-{\Ind_H^G(\eta_A)} & \Ind_H^G\Res_G^H\Ind_H^G(A) \ar[r]^-{\epsilon_{\Ind_H^G(A)}} & \Ind_H^G(A).}
\end{equation}
The map $\kappa\circ \Ind_H^G(\iota)$ gives an inclusion
\[
\xymatrix{\Bigl(C_0(G_{H^{0}})\tensor[_s]{\otimes}{_r}\Bigl[\Bigl(C_0(H)\tensor[_s]{\otimes}{_\rho}A\Bigr)\rtimes_{\mathrm{rt}\otimes\alpha} H\Bigr]\Bigr)\rtimes H \ar[d] \\
 \Bigl(C_0(G)\tensor[_s]{\otimes}{_r}\Bigl[\Bigl(C_0(G_{H^{0}})\tensor[_s]{\otimes}{_\rho}A\Bigr)\rtimes H_{\mathrm{rt}\otimes\alpha}\Bigr]\Bigr)\rtimes G.}
\]
By using the isomorphisms introduced in Lemma \ref{lem:indresisom} above, we can replace the previous inclusion into the more convenient map
\[
\xymatrix{\Bigl(C_0(G_{H^{0}})\tensor[_s]{\otimes}{_{r\otimes\rho}}\Bigl[\Bigl(C_0(H)\tensor[_r]{\otimes}{_\rho}A\Bigr)\rtimes_{\text{rt}\otimes \text{id}} H\Bigr]\Bigr)\rtimes_{} H \ar[d]^-{i} \\
 \Bigl(C_0(\overset{\gamma}{G})\tensor[_r]{\otimes}{_{r}}\Bigl[\Bigl(C_0(\overset{\nu}{G_{H^{0}}})\tensor[_s]{\otimes}{_\rho}A\Bigr)\rtimes \overset{\mu}{H}\Bigr]\Bigr)\rtimes_{\text{rt}\otimes \text{id}} \overset{\eta}{G}.}
\]
Above, the Greek letters indicate our choice of notation for the variable on the given groupoid. These will be useful in a moment. 

Recall the action on $A$ is denoted by $\alpha$. Suppressing notation for the inclusions $H\subseteq G$ and $C_0(H)\subseteq C_0(G)$, the map $i$ can be understood by 
\begin{equation}\label{eq:imap}
i(f)(\eta,\gamma,\mu,\nu)=\alpha_{\nu^{-1}\gamma}(f(\eta,\gamma,\mu,\gamma^{-1}\nu)),
\end{equation}
where $f$ is in $\Gamma_c(H,r^*(C_0(G_{H^{0}})\tensor[_s]{\otimes}{_{r\otimes\rho}}(C_0(H)\tensor[_r]{\otimes}{_\rho}A)\rtimes_{\text{rt}\otimes \text{id}} H))$. Note that the right-hand side is zero unless $\gamma^{-1}\nu\in H$ and $\eta\in H$ (note $\gamma\in G_{H^{0}}$ follows).
The composition in (\ref{eq:unit-counit}) can be computed via the Kasparov product (over the domain of $i$)
\[
[\Ind_H^G\bigl(\bigl(X^H_A\bigr)^{\text{op}}\bigr)]\hatotimes i^*[X^G_{\Ind_H^G A}] .
\]
We claim that
$$i^*[X^G_{\Ind_H^G A}]=\Ind_H^G(X_A^H)$$
The class $i^*[X^G_{\Ind_H^G A}]$ is represented by the Kasparov triple
\[
\left(L^2(G,\Ind_H^G A),(M_{\Ind_H^G A}\rtimes R_G)\circ i,0\right)
\]
while the class $\Ind_H^G(X_A^H)$ is represented by
\[\left(\Ind_H^G L^2(H,A),\Ind_H^G(M_A\rtimes R_H),0\right).\]
Consider the isometric embedding
\[
\Phi\colon\Ind_H^G L^2(H,A) \longrightarrow L^2(G,\Ind_H^G A)
\] from Lemma \ref{Lem:Ind of canonical module}.
We first verify that $\Phi$ intertwines the left actions of $\Ind_H^G \Ind_H^H A$.
To this end recall that for $f\in \Gamma_c(H,r^*(C_0(G_{H^{0}})\otimes_{G^{0}}\Ind_H^H A))$, we have that $(\Ind_H^G(M_A\rtimes R_H)(f)\xi)(g,h,x)$ equals
\[
\int\limits_H\int\limits_H f(h_1,x,h_2,h)\alpha_{h_1}(\xi(h_1^{-1}g,h_1^{-1}hh_2,xh_1))\,d\lambda^{s(h)}(h_2)\,d\lambda^{s(x)}(h_1).
\]

Hence, considering elements $\xi\in \Gamma_c(H,r^*L^2(H,C_0(G_{H^{0}})\otimes_{H^{0}} B))$ and $f\in \Gamma_c(H,r^*(C_0(G)\tensor[_s]{\otimes}{_{r}}(C_0(H)\tensor[_s]{\otimes}{_\rho}A)\rtimes H))$, we compute
\begin{align*}
    &\Phi(\Ind_H^G(M_A\rtimes R_H)(f)\xi)(g,h,x) =\alpha_{x^{-1}g}((\Ind_H^G(M_A\rtimes R_H)(f)\xi)(g^{-1}xh,g^{-1}x,g))\\
    =& \int\limits_H\int\limits_H\alpha_{x^{-1}g}(f(h_1,g,h_2,g^{-1}x)\alpha_{h_1}(\xi(h_1^{-1}g^{-1}xh,h_1^{-1}g^{-1}xh_2,gh_1)))\,d\lambda^{s(x)}(h_2)\,d\lambda^{s(g)}(h_1)\\
    =& \int\limits_H \int\limits_H i(f)(h_1,g,h_2,x)\alpha_{h_2}(\Phi(\xi)(gh_1,h_2^{-1}h,xh_2))\,d\lambda^{s(x)}(h_2)\,d\lambda^{s(g)}(h_1)\\
    =& \left((M_{\Ind_H^G A}\rtimes R_G)(i(f))\Phi(\xi)\right)(g,h,x).
\end{align*}

Since the representation $\Ind_H^G(M_A\rtimes R_H)$ is non-degenerate, it follows immediately that 
 $\text{Img}(\Phi)\subseteq\overline{((M_{\Ind_H^G A}\rtimes R_G)\circ i)L^2(G,\Ind_H^G A)}$. In fact, since $\text{Img}(\Phi)$ is closed, in order to have equality it suffices to show the image is dense. From the definition of $i$ in Eq. \eqref{eq:imap}, we see that 
 \[
\overline{((M_{\Ind_H^G A}\rtimes R_G)\circ i)L^2(G,\Ind_H^G A)}\subseteq \overline{L^2(G_{H^{0}},\Ind_H^G A)\cap F}
\]
where $F$ is spanned by those $L^2$-functions such that $f(g,h,x)=0$ unless $g^{-1}x\in H$ (notation from Lemma \ref{Lem:Ind of canonical module}). With this, the surjectivity is clear from the formula for $\Phi$ in Lemma \ref{Lem:Ind of canonical module}.
Since the element $i^*[X_{\Ind_H^G A}^G]$ can equally well be represented by the submodule $\overline{((M_{\Ind_H^G A}\rtimes R_G)\circ i)L^2(G,\Ind_H^G A)}$ (see \cite[Proposition 18.3.6]{black:kth}) we conclude that $i^*[X_{\Ind_H^G B}^G]=\Ind_H^G X_A^H$ and hence
 \[
[\Ind_H^G\bigl(\bigl(X^H_A\bigr)^{\text{op}}\bigr)]\hatotimes i^*[X^G_{\Ind_H^G A}]=1_{\Ind_H^G A},
\]
as desired.
\end{proof}

\subsection{Compatibility with other functors}

Let $f\colon Y\to X$ be a continuous map, and $A$ and $B$ be $C_0(X)$-algebras. There is a natural isomorphism $f^*(A\otimes_X B)=f^*(A)\otimes_Y f^*(B)$, because both algebras are naturally isomorphic to restrictions of $C_0(Y\times Y)\otimes A \otimes B $ to the same copy of $Y\times X$ in the topological space $Y\times Y\times X\times X$ (cf. \cite[Lemma 6.4]{bondel:kun})

\begin{lemm}\label{lem:indcomp}
There is a natural isomorphism 
\[
\Ind_H^G(A)\otimes_{G^0} B\cong \Ind_H^G(A\otimes_{H^0} \Res_G^H(B)).
\]
In particular $\Ind_H^G\circ f^* \cong f^*\circ\Ind_H^G $. 
\end{lemm}
\begin{proof}
Let $\phi$ be the restriction of the source map to $G_{H^{0}}$. We have 
\[
\phi^*(A\otimes_{H^0} \Res_G^H(B))\cong \phi^*A\otimes_{G_{H^0}} \phi^*\Res_G^H(B)
\]
by the observation above.
Now pushing forward along $\phi$ again we obtain an isomorphism of $H$-$C^*$-algebras
$$(\phi_*(\phi^*(A\otimes_{H^0} \Res_G^H(B)))\cong \phi_*(\phi^*A\otimes_{G_{H^0}} \phi^*\Res_G^H(B))\cong \phi_*\phi^*A\otimes_{H^0} \Res_G^H B$$
Now when we take crossed products by $H$ for the leftmost system, we get $\Ind_H^G (A\otimes_{H^0} \Res_G^H(B))$ by definition.
The rightmost system is just $C_0(G_{H^0})\otimes_{H^0}A\otimes_{H^0} \Res_G^H B$ with the diagonal $H$-action $\mathrm{rt}\otimes \alpha\otimes \Res_G^H(\beta)$. So upon using commutativity of the tensor product and applying Lemma \ref{lem:indresisom}, we may replace it by the action $\mathrm{rt}\otimes \alpha\otimes \mathrm{id}_{B}$.

Summing up, after taking crossed products by $H$ we arrive at the desired conclusion:

\[
\phi_*(\phi^*(A\otimes_X \Res_G^H(B)))\rtimes H \cong \phi_*\phi^*A\rtimes H\otimes_{G^0} B,
\] 
where $\phi_*\phi^*A\rtimes H=\Ind_H^G(A)$ by definition.
\end{proof}

We conclude this section by listing other compatibility relations, which are straightforward as each of them involves a forgetful functor.
\begin{gather*}
\Res_G^H(A\otimes_X B)\cong \Res_G^H(A)\otimes_X \Res_G^H(B)\\
\Ind_H^G\circ f_* \cong f_*\circ \Ind_H^G\qquad \Res_G^H\circ f_* \cong f_*\circ \Res_G^H\qquad \Res_G^H\circ f^* \cong f^*\circ \Res_G^H.
\end{gather*}

\section{The strong Baum--Connes conjecture}\label{sec:bc}

In this section we formulate the strong Baum--Connes conjecture for \emph{étale} groupoids by using the framework developed in the previous section.

As a start, a natural idea is identifying a ``probing'' class of objects $\mathcal{P}r\subseteq \KKK^G$, that we understand somewhat better than a generic object of $\KKK^G$, and for which we can prove the equality of categories $\langle \P r\rangle= \KKK^G$.

\begin{defi}\label{def:proper}

We say that $G$ is \emph{proper} if the \emph{anchor map} $(r,s)\colon G \to X\times X$ is proper.
Furthermore, if $Z$ is a a second countable, locally compact, Hausdorff $G$-space, we say that $G$ \emph{acts properly} on $Z$ if $Z\rtimes G$ is proper. A $G$-algebra $A$ is called proper if there is a proper $G$-space $Z$ such that $A$ is a $Z\rtimes G$-algebra. 

We let $\mathcal{P}r$ denote the class of proper objects in $\KKK^G$.
\end{defi}

Evidently, a commutative $G$-$C^*$-algebra is proper if and only if its spectrum is a proper $G$-space. 

Recall that $G$ is called étale if its source and range maps are local homeomorphisms. A \emph{bisection} is an open $W\subseteq G$ such that $s|_W,r|_W$ are homeomorphisms onto an open in $X$. Hereafter it is assumed that $G$ is étale.

 Recall that a map $f:X \to Y$ is proper at $y \in Y$ if
\begin{itemize} 
\item the fiber at $y$ is compact,
\item any open containing the fiber also contains a tube (a tube is the preimage of an open neighborhood of $y$).
\end{itemize}
A map is proper if and only if it is proper at each point. The proposition below clarifies the local picture of proper actions (cf. \cite[Theorem 4.1.1]{moer:orbi} and \cite[Proposition 2.42]{tu:propnonhaus}).

\begin{prop}\label{prop:prodr}
Suppose $G$ acts properly on $Z$ and denote by $\rho\colon Z \to X$ the moment map. Then for each $z\in Z$ there are open neighborhoods $U^\rho,U$, respectively of $z\in Z$ and $\rho(z)\in X$, satisfying:
\begin{itemize}
\item the fixgroup $\Gamma_z:=\{g\in G\mid g z=z\}$ acts on $U$;
\item There exists an isomorphism from $\Gamma_z \ltimes U$ onto an open subgroupoid $H_z$ of $G|_U$;
\item the $G$-action restricted to $U^\rho$ is induced from $\Gamma_z \ltimes U$, in other words the groupoid $(G\ltimes Z)|_{U^\rho}$ equals $(\Gamma_z \ltimes U)\ltimes U^{\rho}$. 
\end{itemize}
\end{prop}
\begin{proof} 
Since the $G$-action on $Z$ is proper, $\Gamma_z$ is a finite subgroup of the isotropy group $G_{\rho(z)}^{\rho(z)}$. For each $g\in \Gamma_z$ choose an open bisection $W_g$ around $g$. Since $G$ is Hausdorff and $\Gamma_z$ is finite, we may assume that the $W_g$ are pairwise disjoint. 
For any two $g,h\in \Gamma_z$, there is an open neighborhood $V$ of $\rho(z)$ such that $W_{gh}\cap G|_V$ and $(W_g W_h)\cap G|_V$ are non-empty and equal, because both are bisections containing $gh$. Likewise, for each $g$ in $\Gamma_z$ there is an open neighborhood $V$ of $\rho(z)$ where $W_{g^{-1}}\cap G|_{V}$ and $(W_g)^{-1}\cap G|_{V}$ are non-empty and equal. Ranging over the group $\Gamma_z$, we collect a finite number of $V$'s whose intersection we denote by $U$. Notice $U$ is an open neighborhood of $\rho(z)$. We now replace all the $W_g$'s by $W_g\cap r^{-1}(U)\cap s^{-1}(U)$. Then we can define an action of $\Gamma_z$ on $U$
by setting $g\cdot x:=r(s_{|W_g}^{-1}(x))$, i.e. $g$ acts by the partial homeomorphism $U\rightarrow U$ associated with the bisection $W_g$. This is then indeed a well-defined action by the construction of the $W_g$ above. We have a canonical continuous groupoid homomorphism
$$\Phi:\Gamma_z\ltimes U\rightarrow G,\ \Phi(g,x)=s_{|W_g}^{-1}(x).$$
Since the $W_g$ were chosen pairwise disjoint this is in fact an  isomorphism of topological groupoids onto the union $H:=\bigsqcup_{g\in \Gamma_z}W_g$.

Define $U^\prime:=\rho^{-1}(U)$.
Because $G$ acts on $Z$, and $H$ is a subgroupoid of $G$, the notation $U^\prime \rtimes H$ makes sense, and it indicates an open subgroupoid of the restriction $(Z\rtimes G)|_{U^\prime}$. The action of $G$ on $Z$ is proper, in particular the anchor map of the groupoid $Z\rtimes G$ is proper at $z$. Now $U^\prime \rtimes H$ is an open containing the fiber of the anchor map at $z$, therefore it contains a tube. In other words there is an open neighborhood of $z$, say $U^\rho$ (we may assume it is also contained in $U^\prime$), such that the restriction $(Z\rtimes G)|_{U^\rho}$ (i.e., the tube at $U^\rho$) is contained in $U^\prime\rtimes H$. This means that the groupoid that $G$ induces on $U^\rho$ only involves arrows belonging to $H$ (recall that $H$ is isomorphic to $U\rtimes \Gamma$).
\end{proof}

\begin{rema}\label{rem:indprop}
As a simple corollary of Proposition \ref{prop:prodr}, the range map $r\colon s^{-1}(U^\rho)\to Z$ descends to a $G$-equivariant homeomorphism
\begin{equation}\label{eq:inmor}
G\times_{H} U^\rho \to G\cdot U^\rho=V.
\end{equation}
Moreover, the space $s^{-1}(U^\rho)$ provides a principal bibundle implementing an equivalence between $(G\rtimes Z)|_{U^\rho}$ and $(G\rtimes Z)|_{V}$ in the sense of \cite{murewi:morita} (cf. \cite{hoyo:lie}). Hence, the induction functor $\KKK^{(G\rtimes Z)|_{U^\rho}}\to \KKK^{(G\rtimes Z)|_{V}}$ is essentially surjective \cite{gall:kk}, i.e., if $A$ is a $G$-algebra over $Z$ then $A|_{V}$ is isomorphic to $\Ind_{(G\rtimes Z)|_{U^\rho}}^{(G\rtimes Z)|_{V}}(A|_{U^{\rho}})$. We can forget the $C_0(Z)$-structure and obtain $A|_{V}\cong \Ind_{H}^G(A|_{U^{\rho}})$ in $\KKK^G$.
\end{rema}

In Definition \ref{def:proper} for a proper $G$-algebra we can always assume $Z$ to be a realization of $\underbar{E}G$, the classifying space for proper actions of $G$. Indeed if $\phi\colon Z\to \underbar{E}G$ is a $G$-equivariant continuous map, then $\phi^*\colon C_0(\underbar{E}G) \to M(C_0(Z))$ can be precomposed with the structure map $C_0(Z) \to ZM(A)$, making $A$ into an $\underbar{E}G\rtimes G$-algebra.

Note that if $G$ is locally compact, $\sigma$-compact, Hausdorff, $\underbar{E}G$ always exists and is locally compact, $\sigma$-compact, and Hausdorff; in our case $G$ is second countable hence $\underbar{E}G$ is too \cite[Proposition 6.15]{tu:novikov}. 

A subgroupoid of the form $\Phi(\Gamma_z \ltimes U )\subseteq G$, as in Proposition \ref{prop:prodr}, will be called a \emph{compact action} around $\rho(z)$. Given a proper $G$-algebra over $Z=\underbar{E}G$, for any $z\in Z$ we can find an open neighborhood as in \eqref{eq:inmor}. These opens cover $Z$ and we can extract a countable subcover $\mathcal{V}$ (being second countable, $Z$ is a Lindelöf space). Corresponding to this subcover we get a countable collection of compact actions which we denote by~$\F$. Define the full subcategory of \emph{compactly induced objects},
\[
\Cc\I=\{\Ind_Q^G(B) \mid B\in \KKK^Q, Q\in \F \}.
\]
We define a homological ideal $\I$ as the kernel of a single functor
\begin{align}\label{eq:functorF}
F\colon \KKK^G&\to \prod_{Q\in \F}\KKK^Q\\\notag
A&\mapsto (\Res_G^Q(A))_{Q\in \F}
\end{align}
The functor $F$ commutes with direct sums because each restriction functor clearly does. Hence $\I$ is compatible with countable direct sums. The proof below follows the blueprint in \cite[Theorem 7.3]{meyer:tri}, we reproduce it here for completeness.

\begin{theo}\label{thm:bccc}
The projective objects for $\I$ are the retracts of direct sums of objects in $\Cc\I$ and the ideal $\I$ has enough projective objects. Therefore the subcategories in $(\CI,N_\I)$ form a pair of complementary subcategories.
\end{theo}
\begin{proof}
According to \cite[Theorem 3.22]{meyer:tri}, we need to study the (possibly) partially defined left adjoint of the functor $F$ defined in Eq. \eqref{eq:functorF}. Since each compact action $Q\in \F$ is open in $G$, the functor $\Ind_Q^G$ is left adjoint to $\Res_G^Q$. Thus we may take the globally defined adjoint
\[
F^\dagger((A_Q)_{Q\in \F})=\bigoplus_{Q\in \F}\Ind_Q^G(A_Q).
\]
Since $\F$ is countable and $F$ is compatible with countable direct sums, this definition is legitimate. It follows that $\I$ has enough projective objects which are retracts as described. Indeed, $F^\dagger F(A)$ is projective because the isomorphism
\[
\KKK^G(\Ind_Q^G\Res_G^Q(A),B)\cong \KKK^Q(\Res_G^Q(A),\Res_G^Q(B))
\]
is given by $f\mapsto \Res_G^Q(f)\circ\eta_{\Res_G^Q(A)}$, where $\eta$ is the unit of the adjunction. We then see that if $f\in \I$, then we must have $f=0$. Similarly, the counits of the adjunctions yield an $\I$-epic morphism $\delta: F^\dagger F(A)\to A$ \cite[Definition 21]{meyernest:tri}. In particular, if $A$ is already projective, then $\delta$ can be embedded in a split triangle. Split triangles are isomorphic to direct sum triangles \cite[Corollary 1.2.7]{nee:tri}.
\end{proof}

Using notation from Section \ref{subsec:compcell}, and applying the result above, we have $\mathcal{P}=\CI=\langle P_\I\rangle$ and $\mathcal{N}=N_\I$. Since we will only be dealing with the homological ideal $\ker(F)$ just described, we will drop the $\mathcal{I}$ from our notation and just write $\mathcal{N}$ instead of $N_\mathcal{I}$. The objects in $\mathcal{N}\subseteq \KKK^G$ are also referred to as \emph{weakly} contractible.  We denote by $P(A)$ the $\Cc\I$-cellular approximation of $A$. Note $P(A)$ belongs to $\mathcal{P}$.

\begin{coro}
We have the following equivalences,
\begin{equation*}
P(A)\cong P(C_0(G^0))\otimes^{\text{}}_{G^0} A\qquad N(A)\cong N(C_0(G^0))\otimes^{\text{}}_{G^0} A
\end{equation*}
\end{coro}
\begin{proof}
We have already explained that tensorization via the maximal balanced tensor product functor gives a triangulated functor. Hence it maps the canonical exact triangle $P(C_0(G^0))\longrightarrow C_0(G^0)\longrightarrow N(C_0(G^0))$ to an exact triangle $$P(C_0(G^0))\otimes^{\text{}}_{G^0} A\longrightarrow A\longrightarrow N(C_0(G^0))\otimes^{\text{}}_{G^0} A.$$
If we can show that $-\, \otimes^{\text{}}_{G^0} A$ leaves the subcategories $\CI$ and $\mathcal{N}$ invariant, the result follows from the uniqueness statement in Proposition \ref{prop:csub}. Let us begin with the contractible objects: for $B\in \mathcal{N}$, since the restriction functor behaves well with respect to the maximal balanced tensor product, we compute $$\Res_G^Q(\mathrm{id}_{B\otimes_{G^0}^{\text{max}} A})=\Res_G^Q(\mathrm{id}_B)\otimes^{\text{}}_{Q^0}\Res_G^Q(\mathrm{id}_A)= 0,$$ and hence $B\otimes_{G^0}^{\text{max}} A\in \mathcal{N}$.

On the other hand, for every $Q\in \mathcal{F}$ and $B\in\KKK^Q$, Lemma \ref{lem:indcomp} provides $\KKK^G$-equivalences $$\Ind_Q^G(B)\otimes^{\text{}}_{G^0} A\cong\Ind_Q^G(B\otimes^{\text{}}_{Q^0}\Res_G^Q(A))\in \CI.$$
\end{proof}

\begin{defi}\label{def:bc}
We say that $G$ satisfies the \emph{strong} Baum--Connes conjecture (with coefficients in $A$) if the natural map $P(A)\rtimes_r G \to A\rtimes_r G$ is a $\KKK$-equivalence.
\end{defi}

A stronger variant of the formulation above is requiring $P(A)\to A$ to be an isomorphism in $\KKK^G$. However it is known that even the ordinary (weaker) form of the conjecture admits counterexamples \cite{higlafskan:bc}.

We will need the following deep result proved by J.-L. Tu.

\begin{theo}[{\cite{tu:moy}}]\label{thm:tu}
Suppose $G$ is a second countable, locally compact, Hausdorff groupoid. If $G$ acts properly on a continuous field of affine Euclidean spaces, then there exists a proper $G$-$C^*$-algebra $P$ such that $P\cong C_0(G^0)$ in $\KKK^G$. 
\end{theo}
This result has the following immediate consequence:
\begin{coro}
Suppose $G$ is a second countable, locally compact, Hausdorff groupoid. If $G$ admits a proper action on a continuous field of affine Euclidean spaces, then we have the equality of categories $\langle \P r\rangle=\KKK^G$.
\end{coro}
\begin{proof}
If $A\in \KKK^G$ is any $G$-$C^*$-algebra, we have that $A\otimes_{G^0} P$ is proper and $\KKK^G$-equivalent to $A$.
\end{proof}
Our next goal is to show that $\CI=\langle \P r\rangle$.
Let us first treat the proper case:
\begin{lemm}\label{lem:propercase}
Let $G$ be a proper étale groupoid. Then $C_0(G^{0})\in \CI\subseteq \KKK^G$.
\end{lemm}
\begin{proof}
We have to show that $\KKK^G(C_0(G^{0}),N)=0$ for every $\I$-contractible object $N\in\KKK^G$. Since $C_0(G^{0})$ is clearly $C_0(G^{0})$-nuclear we have an isomorphism
$$\KKK^G(C_0(G^{0}),N)\cong E_G(C_0(G^{0}),N)$$
by Corollary \ref{coro:kkenuc}. Consequently, we can work in the setting of $G$-equivariant $E$-theory instead. The upshot is that $E$-theory satisfies excision. In particular, since $G$ is proper, it is locally induced by compact actions as is explained in Proposition \ref{prop:prodr}, i.e., we have a countable cover $\mathcal{V}$ of $G^{0}$ by $G$-invariant sets with $$E_G(C_0(V),N)=\KKK^G(C_0(V),N)=0.$$
As a first step we aim to replace $\mathcal{V}$ by an increasing sequence. In order to arrange this we need to show that given $V_0,V_1\in\mathcal{V}$ we have
$$E_G(C_0(V_0\cup V_1),N)=0$$

Let us first observe that $\KKK^G(C_0(V_0\cap V_1),N)=0$. Following Proposition \ref{prop:prodr} we can write $V_i=GU_i$ such that there exist $H_i\in \mathcal{F}$ with $G\times_{H_i}U_i\cong GU_i=V_i$.
Observe that we have $V_0\cap V_1=G(U_0\cap GU_1)$ and that if $g\in G$ satisfies $g(U_0\cap GU_1)\cap (U_0\cap GU_1)\neq \emptyset$, then also $gU_0\cap U_0\neq \emptyset$ and hence by the construction of $H_0$, $g\in H_0$. Thus,
the canonical map 
$$G\times_{H_0} H_0(U_0\cap GU_1) \to G(U_0\cap GU_1)$$
is a homeomorphism as it is the restriction of the homeomorphism $G\times_{H_0} U_0\cong GU_0$.
It follows that $C_0(V_0\cap V_1)\cong C_0(G\times_{H_0} H_0(U_0\cap GU_1))=\Ind_{H_0}^G(C_0(H_0(U_0\cap GU_1))\in \mathcal{CI},$
and hence $E_G(C_0(V_0\cap V_1),N)\cong \KKK^G(C_0(V_0\cap V_1),N)=0$.

The corresponding statement for the union $V_0\cup V_1$ now follows easily from the long exact sequences in $E_G$-theory associated with the short exact sequences:
\begin{align*}
&0\longrightarrow C_0(V_0\cap V_1)\longrightarrow C_0(V_1)\longrightarrow C_0(V_1\smallsetminus V_0)\longrightarrow 0,\\
&0\longrightarrow C_0(V_0)\longrightarrow C_0(V_0\cup V_1)\longrightarrow C_0(V_1\smallsetminus V_0)\longrightarrow 0.
\end{align*}
In each sequence two out of three groups in the induced long exact sequence vanish and hence so does the third. Replacing $V_n$ by $\bigcup_{i=1}^n V_i$ we can assume that $\mathcal{V}=(V_n)_{n\in\N}$ is an increasing sequence. We clearly have $C_0(G^{0})=\varinjlim_n C_0(V_n)$ and since $E$-theory has countable direct sums we have a Milnor $\lim ^1$-sequence (see Lemma \ref{lem:millimone})
$$0\longrightarrow \varprojlim{}^1 E_G(C_0(V_n),\Sigma N)\longrightarrow E_G(C_0(G^{0}),N)\longrightarrow \varprojlim E_G(C_0(V_n), N)\longrightarrow 0$$
Since the left and right terms are both zero, this concludes the proof.
\end{proof}
In the argument above we can replace $C_0(G^{0})$ by any $\KKK^G$-nuclear $G$-algebra $A$. 

\begin{theo}\label{theo:ci=pr}
The localizing subcategory of $\KKK^G$ generated by compactly induced objects equals the one generated by proper objects, i.e., $\CI=\langle \P r\rangle$.
\end{theo}
\begin{proof}

Consider the canonical triangle
\begin{equation}
    P\stackrel{D}{\longrightarrow}C_0(G^{0})\stackrel{\eta}{\longrightarrow}N \longrightarrow \Sigma P,
\end{equation}
and let $p:G\ltimes \E G\rightarrow G$ denote the projection homomorphism. The associated functor $p^*:\KKK^G\rightarrow \KKK^{G\ltimes\E G}$ maps contractible objects to contractible objects.
Indeed, since $\E G$ is a proper $G$-space, a compact action for $G\ltimes\E G$ is just given by the restriction to one of the sets $U^\rho$ as in Proposition \ref{prop:prodr}. Continuing to use the notation from that proposition let $Q$ be the open copy of $\Gamma_z\ltimes U$ inside $G$, a compact action for $G$! Then the compositions of groupoid homomorphisms $(G\ltimes \E G)|_{U^\rho}\hookrightarrow G\ltimes \E G\stackrel{p}{\rightarrow}G$ and $(G\ltimes \E G)|_{U^\rho}\cong (\Gamma_z\ltimes U)\ltimes U^\rho\stackrel{p}{\rightarrow}\Gamma_z\ltimes U\cong Q\hookrightarrow G$ coincide. The resulting commutative diagram of $\KKK$ groups gives $\Res_{(G\ltimes \E G)|_{U^\rho}}(\mathrm{id}_{p^*N})=\Res_{(G\ltimes \E G)|_{U^\rho}}(p^*(\mathrm{id}_N))=p^*(\Res_Q(\mathrm{id}_N))=0$ for any contractible object $N\in \KKK^G$.

Combining this with Lemma \ref{lem:propercase}, we can use the fact that $\CI$ and $\mathcal{N}$ are complementary to conclude that
$p^*(\eta)\in \KK[][G\ltimes \E G]{C_0(\E G)}{p^*N}=0$.

Now let $A\in\KKK^G$ be an arbitrary proper $G$-algebra. As explained before we may assume that $A$ is a $C_0(\E G)$-algebra.
From our observation above it follows that $p^*(\eta)\otimes_{\E G} 1_A=0$.
Since the functors $p^*$ and $\sigma_A$ are both triangulated, we can apply them in this order to obtain a triangle
$$p^*P\otimes_{\E G} A\rightarrow C_0(\E G)\otimes_{\E G}A\rightarrow p^*N\otimes_{\E G}A\rightarrow \Sigma (p^*P\otimes_{\E G} A).$$
Note, that $C_0(\E G)\otimes_{\E G}A\cong A$.
Rotating this triangle gives the triangle

$$p^*N\otimes_{\E G}A\longrightarrow \Sigma (p^*P\otimes_{\E G}A)\longrightarrow \Sigma A\stackrel{0}{\longrightarrow}\Sigma (p^*N\otimes_{\E G}A),$$
in which the last morphism is zero as indicated.
Thus, \cite[Corollary 1.2.7]{nee:tri} implies that the latter triangle splits, namely $\Sigma (p^*P\otimes_{\E G}A)\cong (p^*N\otimes_{\E G}A)\oplus \Sigma A$.

In particular, after suspending once more we obtain a retraction $A\longrightarrow p^*P\otimes_{\E G}A$, i.e. a right inverse of $p^*D\otimes_{\E G} 1_A:p^*P\otimes_{\E G}A\longrightarrow A$.
Now applying the forgetful functor $p_{*}$ gives a retraction $A\longrightarrow p_{*}(p^*P\otimes_{\E G}A)\cong P\otimes_{G^0} A\cong P(A)$. Since $\CI$ is a thick subcategory of $\KKK^G$ it follows that $A\in \CI$.
\end{proof}

\begin{rema}\label{rem:propsub}
In general, we do not know if any object in $\CI$ is equivalent in $\KKK^G$ to a proper $G$-$C^*$-algebra. However, if the cellular approximation $P=P(C_0(G^0))$ happens to be proper (e.g., in the setting of Theorem \ref{thm:tu}) then the previous statement clearly holds, because for any $A\in \CI$, we have that $P\otimes_{G^0} A\cong A$ is a proper $G$-$C^*$-algebra (cf. \cite[Corollary 4.37]{meyereme:dualities} and \cite[Section 7]{meyernest:tri}.)
\end{rema}

The corollary below identifies the localization category in terms of the more classical $\RKK^G$-functor.  Recall a morphism $f\colon A\to B$  in $\KKK^G$ is called a \emph{weak} equivalence if $F(f)$ is an isomorphism, where $F$ is the functor in \eqref{eq:functorF}. For instance, the natural map $D_A\colon P(A)\to A$ is a weak equivalence.

\begin{theo}\label{thm:rkkloc}
Let $p\colon \E G\to G^{0}$ be the moment map underlying the $G$-action. The functor $p^*:\KKK^G\to\RKK^G(\E G)$ is an isomorphism of categories up to localization at $N_\I$. More precisely, the indicated maps in the following commutative diagram are isomorphisms.
\[
\xymatrixcolsep{5em}\xymatrix{ \KKK^G(P(A),B)  \ar[d]_-{p^*}^{\cong} & \KKK^G(A,B) \ar[l]_-{D_A^*}\ar[d]^-{p^*}\\
\RKK^G(\E G; P(A),B) &\RKK^G(\E G; A,B)\ar[l]^-{p^*(D_A)^*}_{\cong}}
\] 
\end{theo}
\begin{proof}
Let us first consider the bottom map. Since $\RKK^G(\E G; -\,,B)$ is a cohomological functor, the claim follows from the inclusion $N_\I \subseteq\mathrm{ker}(p^*)$. If $A$ is weakly contractible, then $p^*(A)$ is both weakly contractible and proper, hence $\KKK^{G\ltimes \E G}(p^*A,p^*A)=0$ by Proposition \ref{prop:csub}. Thus $p^*(A)=0$.

Secondly, let us turn to the vertical map. Both the top and the bottom groups are functorial in the first slot and compatible with direct sums, hence the class of objects for which $p^*$ is an isomorphism is localizing. Thus we can assume $P(A)=\Ind_H^G(D)$ for some compact action $H\subseteq G$. Then, by using the Induction-Restriction adjunction and exchanging $p^*$ and $\Ind_H^G$, we can reduce ourselves to proving that
\begin{equation}\label{eq:pisom}
p^*\colon \KKK^H(D,\Res_G^H(B))\to \RKK^H(U^\rho; D,\Res_G^H(B))
\end{equation}
is an isomorphism (we are using notation from  Equation \eqref{eq:inmor}). The subgroupoid $H$ is a compact action and it satisfies a strong form of the Baum--Connes conjecture, in particular it admits a Dirac-dual-Dirac triple as in \cite[Definition 4.38]{meyereme:dualities}. Then \cite[Theorem 4.34 \& 4.39]{meyereme:dualities} imply that Eq. \eqref{eq:pisom} is an isomorphism. More concretely, if $P^\prime$ a proper $C^*$-algebra which is also the cellular approximation of $C(H^{(0)})$, then the inverse map is given by $[x]\mapsto p_*(P^\prime\otimes_{U^\rho} [x])$ (cf. \cite[Lemma 4.31]{meyereme:dualities}).
\end{proof}

\begin{rema}
The second part of the proof above should be viewed as a statement about the $H$-equivariant ``contractibility'' of $\E G$ (cf. \cite[Theorem 7.1]{nestmeyer:loc} and \cite[Theorem 11.3]{tu:moy}). 
Concerning the map in Eq. \eqref{eq:pisom}, if the $G$-cellular approximation $P$ was $\KKK^G$-equivalent to a proper $C^*$-algebra, then the map $[x]\mapsto p_*(P\otimes_{\E G} [x])$ would provide an inverse already in $\KKK^G$. This holds for many groupoids, as is showed by Theorem \ref{thm:tu}, however by passing to $H$ via the adjunction, we do not need to assume that $P$ is proper in the theorem above.
\end{rema}

The relation to the ordinary Baum--Connes conjecture is explained by means of the following result (compare with \cite[Theorem 6.12]{meyereme:dualities}, see also \cite{nestmeyer:loc} for action groupoids). The left-hand side of the Baum--Connes assembly map (with coefficients in $A$) is often denoted $K_*^{\mathrm{top}}(G;A)$ and is defined as $\varinjlim_{Y \subseteq \E G} \KKK^G(C_0(Y),A)$, the limit ranging over the directed set of $G$-invariant $G$-compact subspaces of $\E G$.

\begin{theo}\label{thm:asweak}
Let $A\in \KKK^G$ be a $G$-$C^*$-algebra and denote by $\mu^G_A$ the associated assembly map. Let $D_A\colon P(A)\to A$ be the natural $\KKK^G$-morphism. The indicated maps in the following commuting diagram are isomorphisms.
\begin{equation}\label{eq:bcmap}
\xymatrixcolsep{3.5em} \xymatrix{ K_*^{\mathrm{top}}(G;A) \ar[d]_-{\cong} \ar[r]^-{\mu^G_A} & K_*(A\rtimes G)\\
K_*^{\mathrm{top}}(G;P(A)) \ar[r]^-{\mu^G_{P(A)}}_{\cong} & K_*(P(A)\rtimes G) \ar[u]_-{D_A\rtimes G} &}
\end{equation}
\end{theo}
\begin{proof}
The functor $K_*^{\mathrm{top}}(G;\,-)$ is homological, it commutes with direct sums, and by the vertical isomorphism in Theorem \ref{thm:rkkloc}, it is functorial for maps in $\RKK^G(\E G; A,B)$ . The same theorem also implies $p^*(D_A)$ is invertible, thus the left map in the diagram above is an isomorphism. Now $\mu^G_{P(A)}$ is an isomorphism if the Baum--Connes conjecture holds for compactly induced coefficient algebras. This is proved in \cite{ce:perma} (see also \cite{ceo:shapiro} and \cite[Theorem 4.48]{meyereme:dualities}). 
\end{proof}

Combining $\CI=\langle \P r\rangle$ and Tu's Theorem \ref{thm:tu}, we obtain the following.

\begin{coro}
Suppose $G$ is a second countable, locally compact, Hausdorff groupoid. Assume that there exists a proper $G$-$C^*$-algebra $P$ such that $P\cong C_0(G^0)$ in $\KKK^G$. Then $G$ satisfies the strong Baum--Connes conjecture with coefficients.
\end{coro}

The previous corollary applies in particular to all amenable groupoids and more generally to all \textit{a-T-menable} groupoids (a-T-menability is also known as the \textit{Haagerup property}) by \cite[Proposition~3.8]{tu:moy}. 

The following lemma shows that we can use Theorem \ref{theo:ci=pr} to rephrase the definition of $\mathcal{N}$ as the category of contractible objects with respect to the kernel of the joint restriction functor to all proper open subgroupoids (instead of just the compact actions).
\begin{lemm}\label{lem:contractible objects}
Let $B\in \KKK^G$. Then $B\in \mathcal{N}$ if and only if $\Res_G^H(\mathrm{id}_B)=0$ for all proper open subgroupoids $H\subseteq G$.
\end{lemm}
\begin{proof}
Suppose that $B\in\mathcal{N}$. By Theorem \ref{theo:ci=pr} and the fact that $(\CI,\mathcal{N})$ is a pair of complementary subcategories, we get that $\KKK^G(A,B)=0$ for all $A\in \P r$. If $H\subseteq G$ is a proper open subgroupoid, then $\Ind_H^G D\in \P r$ for all $D\in \KKK^H$. Using the induction-restriction adjunction, we get that
$$\KKK^H(D,\Res_G^H B)\cong \KKK^G(\Ind_H^G D,B)=0$$
for all $D\in \KKK^H$. If we apply this to $D=\Res_G^H(B)$ we get, in particular, that $\Res_G^H(\mathrm{id}_B)=\mathrm{id}_{\Res_G^H B}=0$.
The converse follows from the definition of $\mathcal{N}$ and the fact that each $Q\in\mathcal{F}$ is a proper open subgroupoid of $G$.
\end{proof}

\section{Applications}\label{sec:app}

\subsection{The UCT}
The article \cite{bondel:kun} established a connection between the Baum--Connes conjecture for groupoids and the Künneth formula for groupoid crossed products.
Now the Universal Coefficient Theorem (UCT) introduced in \cite{rosscho:kunneth} is formally stronger than the Künneth formula, so philosophically speaking it may not come as a surprise that a similar relation exists between the strong Baum--Connes conjecture and the UCT.
\begin{prop}
Let $(A,G,\alpha)$ be a groupoid dynamical system with $A$ type I. Then $P(A)\rtimes_r G$ satisfies the UCT. If furthermore $G$ satisfies the strong Baum--Connes conjecture, then $A\rtimes_r G$ satisfies the UCT.
\end{prop}
\begin{proof}
If $A$ is a type I $\mathrm{C}^*$-algebra and $H$ is a proper groupoid, the crossed product $A\rtimes H$ is type I by \cite[Proposition~10.3]{tu:moy}. Given $A$ as in the claim, and $H\subseteq G$ a proper open subgroupoid acting on $A$, then $C_0(G_{H^0})\otimes A$ is type I, $C_0(G_{H^0})\otimes_{H^0} A$ is type I (because it is a quotient), and $L_H(A):= \Ind_H^G\Res_G^H(A)$ is type I as well. Hence $L_H(A)$ belongs to the bootstrap class. Since $L_H(A)\rtimes_r G$ is Morita equivalent to  $A\rtimes_r H$ and $P(A)\rtimes_r G$ belongs to the localising subcategory of $\KKK$ generated by
\[
\{ L_{H_1}\cdots L_{H_n}(A)\rtimes G \mid n\in \N, H_i\subseteq G \text{ proper and open}\},
\]
it follows that $P(A)\rtimes_r G$ belongs to the bootstrap class as well. 

Since the bootstrap class is closed under $\KKK$-equivalence, the strong Baum--Connes conjecture yields the result.
\end{proof}
We do in particular obtain the following corollary, generalising \cite{bali:CartanUCT, kls:haagerup}. To state it, recall that a twist over $G$ is a central extension
$$G^0\times\mathbb{T}\rightarrow \Sigma\stackrel{j}{\rightarrow} G,$$
and that one can associate the twisted groupoid $C^*$-algebra $C_r^*(G,\Sigma)$ to this data (see \cite{ren:cartan} for the details of this construction).
\begin{coro}
Let $\Sigma$ be a twist over an étale groupoid $G$. If $G$ satisfies the strong Baum--Connes conjecture, then $C_r^*(G,\Sigma)$ satisfies the UCT.
\end{coro}
\begin{proof}
Apply the stabilisation trick \cite[Proposition~5.1]{erpwill} to replace $C_r^*(G,\Sigma)$ up to Morita-equivalence by $K(H)\rtimes_r G$, where $K(H)$ denotes the algebra of compact operators on a suitable Hilbert $C_0(G^0)$-module. As $K(H)$ is type I, the previous proposition applies.
\end{proof}

\subsection{The Going-Down principle}

We generalize some results obtained by the first author for ample groupoids \cite{bon:goingdownbc} to the general étale case.

\begin{theo}\label{thm:cigdp}
Suppose there is an element $f\in \KKK^G(A,B)$ such that 
\[ 
\KKK^H(D,\Res_G^H(A))\xrightarrow{-\,\hot \Res_G^H(f)\,}\KKK^H(D, \Res_G^H(B))
\]
is an isomorphism for all $H\in \F$ and separable $H$-$C^*$-algebras $D$. Then $f$ is a weak equivalence, and in particular the Kasparov product induces an isomorphism
\begin{equation}\label{eq:cons}
 - \,\hot_A\, f:K_*^{\mathrm{top}}(G;A)\rightarrow K_*^{\mathrm{top}}(G;B).
\end{equation}
\end{theo}
\begin{proof}
Using the Induction-Restriction adjunction the hypothesis is equivalent to the following map being an isomorphism for any $\tilde{D}\in \Cc\I$,
\begin{equation*}
\KKK^G(\tilde{D}, A)\xrightarrow{-\,\hot f\,}\KKK^G(\tilde{D}, B).
\end{equation*}

Applying the functor $\KKK^G(\tilde{D}, -)$ to a mapping cone triangle for $f$, and using the Five Lemma we deduce that $\KKK^G(\tilde{D}, \Cone(f))\cong 0$ for all $\tilde{D}$ in $\CI$. Now by Theorem \ref{prop:csub} we get $\Cone(f)\in N_\I$. The rest follows from Theorems \ref{thm:rkkloc} and \ref{thm:asweak}.
\end{proof}

If we are only interested in studying the assembly map, then we might want to prove Equation \eqref{eq:cons} without necessarily proving that $A$ and $B$ have isomorphic cellular approximations. The following result is a version of the previous one ``after $K_*(-\rtimes G)$'', and it can be proved with slightly weaker assumptions. 

\begin{theo}[{\normalfont cf. \cite[Theorem 7.10]{bon:goingdownbc}}] \label{thm:GD}
    Let $f\in \mathrm{KK}^G(A_1,A_2)$ be an element such that the induced map
    $$K_*(\jmath_H(\mathrm{Res}^H_G(f))):K_*(\mathrm{Res}^H_G(A_1)\rtimes H)\rightarrow K_*(\mathrm{Res}^H_G(A_2)\rtimes H)$$
    is an isomorphism for all proper open subgroupoids $H\subseteq Q$ for all $Q\in \mathcal{F}$.
    Then $$K_*(\jmath_G(P(f))):K_*(P(A_1)\rtimes_r G)\rightarrow K_*(P(A_2)\rtimes_r G)$$
    is an isomorphism.
\end{theo}
The proof requires some preparation.
For a subgroupoid $H\subseteq G$ let $L_H:= \mathrm{Ind}_H^G\circ \mathrm{Res}_G^H$. Consider the class $\mathcal{P}_0$ of $G$-algebras of the form $(L_{H_n}\circ\cdots\circ L_{H_1})(C_0(G^0))$ for $n\in \mathbb N$ and $H_i\in \mathcal{F}$.

\begin{lemm}
    $P(C_0(G^0))\in \langle \mathcal{P}_0\rangle$.
\end{lemm}
\begin{proof}
    By \cite[Proposition~3.18]{meyer:tri} the $\mathcal{CI}$-cellular approximation $P(C_0(G^0))$ can be computed as the homotopy limit of a phantom castle over $C_0(G^0)$. Hence it is enough to show that such a phantom castle can be found inside $\langle \mathcal{P}_0\rangle$.
    Using the fact that $\langle \mathcal{P}_0\rangle$ is localising, an inspection of the construction of such a phantom castle in \cite{meyer:tri} shows that it suffices to show that $C_0(G^0)$ admits a projective resolution by objects in $\langle \mathcal{P}_0\rangle$. The standard way to construct such a projective resolution is by considering the algebras $(F^\dagger\circ F)^n(C_0(G^0))$ for $n\geq 1$. 
    
    We will prove that this resolution is contained in $\langle \mathcal{P}_0\rangle$ by induction. First, we have $(F^\dagger\circ F)(C_0(G^0))=\bigoplus_{H\in \mathcal{F}} \mathrm{Ind}_H^G \mathrm{Res}_G^H C_0(G^0)\in \langle \mathcal{P}_0\rangle$. 
    Assuming now that the claim holds for $n-1$, we compute
    $$(F^\dagger\circ F)^n(C_0(G^0))=\bigoplus_{H\in \mathcal{F}} \mathrm{Ind}_H \mathrm{Res}_H((F^\dagger\circ F)^{n-1}(C_0(G^0))),$$
    and the latter is contained in $\langle \mathcal{P}_0\rangle$ since $L_H(\langle \mathcal{P}_0\rangle)\subseteq \langle \mathcal{P}_0\rangle$ (we have $L_H(\mathcal{P}_0)\subseteq \mathcal{P}_0$ by definition of $\mathcal{P}_0$ and hence the general statement follows from the fact that $L_H$ is triangulated and compatible with direct sums).
\end{proof}

\begin{proof}[Proof of Theorem \ref{thm:GD}]
We will show that
\begin{equation}\label{eq:gdeq}
    K_*((\jmath_G(\mathrm{id}_B\otimes_{G^0} f))):K_*((B\otimes_{G^0} A_1)\rtimes G)\rightarrow K_*((B\otimes_{G^0} A_2)\rtimes G)
\end{equation}
is an isomorphism for all $B\in \mathcal{P}_0$. Once this is proven, we can complete the proof as follows:
since $K$-theory is a homological functor (compatible with direct sums), these isomorphisms imply that \eqref{eq:gdeq} is also an isomorphism for $B\in \langle \mathcal{P}_0\rangle$ by a routine argument involving the Five Lemma. 

In particular, we can take $B=P(C_0(G^0))$ by the previous Lemma. Noting further that $P(A)\rtimes G\cong P(A)\rtimes_r G$ in $\mathrm{KK}$, the proof will be complete.
Thus, in what follows we show that \eqref{eq:gdeq} is an isomorphism for all $B\in \mathcal{P}_0$.

\noindent\textbf{Step 1:} 
We will first prove that \eqref{eq:gdeq} is an isomorphism for $B=L_H(C_0(G^0))=C_0(G/H)$ whenever $H\subseteq Q$ for some $Q\in \mathcal{F}$.. In this case we have natural $G$-equivariant isomorphisms
$$B\otimes_{G^0}A_i\cong \mathrm{Ind}_H^G(C_0(H^0))\otimes_{G^0} A_i\cong \mathrm{Ind}_H^G(\mathrm{Res}_G^H(A_i))$$
and hence $(B\otimes_{G^0} A_i)\rtimes G$ is Morita equivalent to $\mathrm{Res}_G^H(A_i)\rtimes H$. Thus, this case follows directly from the assumption.

\noindent\textbf{Step 2:} Suppose $B=L_K(C_0(X))=\mathrm{Ind}_K^G C_0(X|_{K^0})$, where $X$ is any second countable proper \'etale $G$-space with anchor map $p:X\rightarrow G^0$, and $K\in \mathcal{F}$. We claim that $\eqref{eq:gdeq}$ is an isomorphism for this choice of $B$.
 Let $\mathcal{B}$ be a countable basis for the topology of $X|_{K^0}$ consisting of open subsets of $X|_{K^0}$ on which $p$ restricts to a homeomorphism. Then we can write 
    $$X|_{K^0}=\bigcup_{S\in \mathcal{B}} KS.$$
    Since $\mathcal{B}$ is countable we may enumerate its elements writing $\mathcal{B}=\{S_n\mid n\in \mathbb N\}$. Let $X_n:=\bigcup_{i=1}^n KS_n$. Then $X_n$ is an open $K$-invariant subset of $X$. Moreover, $C_0(X|_{K^0})=\varinjlim_n C_0(X_n)$ where the connecting maps are just given by the canonical inclusions. Since the induction functor, tensor products and the maximal crossed product as well as $K$-theory are all compatible with inductive limits, it suffices to show that \eqref{eq:gdeq} is an isomorphism for $B=\mathrm{Ind}_K^G C_0(X_n)$.
    We will do this by induction on $n$. 
    
    For $n=1$ observe that for every $S\in \mathcal{B}$ there are identifications $KS\cong K\times_{\mathrm{Stab}(S)} S$ where $\mathrm{Stab}(S)$ is the proper open subgroupoid of $K$ defined as $\mathrm{Stab}(S)=\{g\in K\mid gS\subseteq S\}$. Note that the restriction of the anchor map induces a homeomorphism $S\cong \mathrm{Stab}(S)^0$. It follows that
    $$C_0(KS)\cong C_0(K\times_{\mathrm{Stab}(S)} S)\cong \mathrm{Ind}_{\mathrm{Stab}(S)}^K(C_0(\mathrm{Stab}(S)^0)).$$
    and using induction in stages we conclude that 
    $$
    \mathrm{Ind}_K^G C_0(KS)=\mathrm{Ind}_{\mathrm{Stab}(S)}^G C_0(\mathrm{Stab}(S)^0)=C_0(G/\mathrm{Stab}(S)).$$
    Since $\mathrm{Stab}(S)$ is a proper open subgroupoid of $K\in \mathcal{F}$, it follows that \eqref{eq:gdeq} is an isomorphism for $B=\mathrm{Ind}_K^G C_0(KS)$ by Step 1 above.
    
    Next, consider a union $KS\cup KT$ for $S,T\in \mathcal{B}$.
    Then we have two short exact sequences of $K$-algebras
    $$0\rightarrow C_0(KS\cap KT)\rightarrow  C_0(KS)\rightarrow  C_0(KS\setminus KT)\rightarrow 0$$
    and
    $$0\rightarrow C_0(KT)\rightarrow  C_0(KS\cup KT)\rightarrow  C_0(KS\setminus KT)\rightarrow 0$$
Using that the functors $\mathrm{Ind}_K^G -$, $(-\, \otimes_{G^0} A_i)$, and $(-\,\rtimes G)$ are all exact, we can apply them (in this order) to the above sequences and the result remains exact. Hence we obtain induced six-term exact sequences in $K$-theory, which can be compared using the maps induced by $f$.
Thus, using the case $n=1$ above, to prove the claim for the union $KS\cup KT$ for $S,T\in \mathcal{B}$, it suffices to prove it for $KS\cap KT$.
To this end note that
$$KS\cap KT=K(S\cap KT).$$
Considering the subgroupoid $\mathrm{Stab}(S\cap KT)$ of $K$ defined as above we can employ the same arguments as in the case $n=1$ to conclude that $C_0(KS\cap KT))\cong \mathrm{Ind}_{\mathrm{Stab}(S\cap KT)}^K(C_0(\mathrm{Stab}(S\cap KT)^0))$, and hence using induction in stages again, we conclude that \eqref{eq:gdeq} is an isomorphism for
$$B=\mathrm{Ind}_K^G C_0(KS\cap KT)=\mathrm{Ind}_{\mathrm{Stab}(S\cap KT)}^G C_0(\mathrm{Stab}(S\cap KT)^0)\cong C_0(G/\mathrm{Stab}(S\cap KT)).$$

Inductively, we can continue in this way to prove the isomorphism in line \eqref{eq:gdeq} for all $B=\mathrm{Ind}_K^G C_0(X_n)$ and hence complete step 2 by passing to the inductive limit.

\noindent\textbf{Step 3:}
We can now prove that \eqref{eq:gdeq} is an isomorphism for all $B\in \mathcal{P}_0$ by induction. The base case is contained in Step 1 above. For the induction step note that $L_{H_n}\cdots L_{H_1}(C_0(G^0))\cong C_0(G/H_n\times_{G^0}\ldots\times_{G^0}G/H_1)$ and observe that the space $X:=G/H_n\times_{G^0}\ldots\times_{G^0}G/H_1$ is an \'etale proper $G$-space. Thus we can just apply Step 2 to complete the proof.
\end{proof}

This result directly allows to generalize several results obtained by the first author for ample groupoids to the general étale case.
\subsubsection{Homotopies of twists}
Let $G$ be an étale groupoid. A \emph{homotopy of twists} is a twist over $G\times [0,1]$, i.e. a central extension of the form
$$G^0\times [0,1]\times\mathbb{T}\rightarrow \Sigma\stackrel{j}{\rightarrow} G\times [0,1].$$

\begin{theo}
Let $G$ be a second countable étale groupoid satisfying the Baum--Connes conjecture with coefficients. If $\Sigma$ is a homotopy of twists over $G$, then for each $t\in[0,1]$ the canonical map $q_t:C_r^*(G\times [0,1],\Sigma)\rightarrow C_r^*(G,\Sigma_t)$ induces an isomorphism in K-theory.
\end{theo}
\begin{proof}
The idea of the proof is the same as for the main result in \cite{B21}: using a groupoid version of the Packer-Raeburn stabilisation trick and the Going-Down principle (Theorem \ref{thm:GD}) one only has to prove the result for all proper open subgroupoids of all elements $H\in \mathcal{F}$ in place of $G$. Recall that all the groupoids $H\in \mathcal{F}$ are (isomorphic to) transformation groupoids of finite groups. Hence, if the original homotopy of twists over $G$ is topologically trivial in the sense that the map $j$ has a continuous section (this means that the twist is equivalent to a continuous $2$-cocycle), one can apply an earlier result of Gillaspy \cite{gill:twist} to finish the proof. 
In the setting of ample groupoids treated in \cite{B21} the requirement that the twist is topologically trivial is not actually a restriction by \cite[Proposition~4.2]{B21}.

In the \'etale setting twists are no longer automatically topologically trivial, so instead we use a refinement of the Going-Down principle. Observe that the constructions and results from the previous section allow some flexibility in choosing the family $\mathcal{F}$ of subgroupoids of $G$. Indeed, if $\mathcal{F}'$ is another family of subgroupoids of $G$ with the property that every proper action of $G$ is locally induced by members of $\mathcal{F}'$, we can replace $\mathcal{F}$ by $\mathcal{F}'$ in all the results of Section 3 and hence also in Theorem \ref{thm:cigdp}.

Now given a homotopy of twists with quotient map $j: \Sigma\rightarrow G\times [0,1]$ we claim that there exists a family $\mathcal{F}'$ of compact actions for $G$ as above with the additional property that the restricted twist $j^{-1}(H\times [0,1])\rightarrow H\times [0,1]$ (this is now a homotopy of twists over $H$) admits a continuous cross section. 

Let us explain how this works: by the proof of \cite[Proposition~4.2]{B21} every $g\in G$ admits an open neighbourhood $V$ such that there exists a local section $V\times [0,1]\rightarrow \Sigma$ of $j$. Now given a proper action of $G$ we will proceed as in the proof of Proposition \ref{prop:prodr}, but (in the notation of that proof) we additionally choose the bisections $W_g$ to be the domains of local sections of $j$ as above. 
Since the $W_g$ can be assumed to be pairwise disjoint and the remaining construction in the proof of Proposition \ref{prop:prodr} just shrinks them further, we can patch the resulting finitely many local sections $W_g\times [0,1]\rightarrow \Sigma$ together to obtain the desired continuous section $H\times [0,1]\rightarrow \Sigma$.
Since $H$ is of the form $\Gamma\ltimes U$ for a finite group $\Gamma$ and an open subset $U\subseteq G^0$ we are again the position to apply Gillaspy's result to conclude that $q_t$ induces an isomorphism for all $H\in \mathcal{F}'$.
To lift the result from this to all of $G$ one can follow the arguments in \cite{B21} again.
\end{proof}
\subsection{Amenability at infinity}
Recall that a locally compact Hausdorff groupoid $G$ is called \textit{amenable at infinity}, if there exists a $G$-space $Y$ with proper momentum map $p:Y\rightarrow G^{0}$, and such that $G\ltimes Y$ is (topologically) amenable. 

It is called \textit{strongly amenable at infinity} if in addition, the momentum map $p$ admits a continuous cross section. Since $p$ is a proper map, it induces an equivariant $\ast$-homomorphism $C_0(G^0)\rightarrow C_0(Y)$ and can hence be viewed as a morphism
$$\mathbf{p}\in \KKK^G(C_0(G^{0}),C_0(Y)).$$
It was shown in \cite[Lemma~4.9]{ad:exact} that if $G$ is strongly amenable at infinity, then the space $Y$ witnessing this can be chosen second countable. Replacing this space further by the space of probability measures on $Y$ supported in fibres we may also assume that each fibre (with respect to $p$) is a convex space and that $G$ acts by affine transformations. The following result is \cite[Proposition~8.2]{bon:goingdownbc}:
\begin{prop}\label{prop:aaiproper}
Let $G$ be a second countable \'etale groupoid and let $Y$ be a fibrewise convex space on which $G$ acts by affine transformations. Suppose further that the anchor map $p:Y\rightarrow G^{0}$ admits a continuous cross section. If $H\subseteq G$ is a proper open subgroupoid, then the restriction of $p$ to $p^{-1}(H^{0})$ is an $H$-equivariant homotopy equivalence.
In particular, $\Res_G^H(\mathbf{p})\in \KKK^H(C_0(H^{0}),C_0(p^{-1}(H^{0}))$ is invertible.
\end{prop}

We obtain the following consequence:
\begin{theo}
Let $G$ be a second countable \'etale groupoid which is strongly amenable at infinity. Then there exists an element $\eta\in \KKK^G(C_0(G^0),P(C_0(G^0)))$ such that $\eta\circ D=\mathrm{id}_{P(C_0(G^0))}$, where $D$ denotes the Dirac morphism for $G$. In particular, the Baum--Connes assembly map $\mu_A$ for $G$ is split injective for all $A\in\KKK^G$.
\end{theo}
\begin{proof}
It follows immediately from Theorem \ref{thm:cigdp} and Proposition \ref{prop:aaiproper} that $p\in \KKK^G(C_0(G^{0}),C_0(Y))$ is a weak equivalence. Hence $P(p)$ is an isomorphism in $\KKK^G$. Moreover, since $G$ acts amenably on $Y$, the natural morphism $D_{C_0(Y)}:P(C_0(Y))\rightarrow C_0(Y)$ is an isomorphism in $\KKK^{G\ltimes Y}$. Consider the canonical forgetful functor $p_*:\KKK^{G\ltimes Y}\rightarrow \KKK^G$ induced by the anchor map $p:Y\to G^0$.
It is not hard to see that $p_*$ is a triangulated functor. Moreover, it maps proper objects to proper objects (if $Z$ is a proper $G\ltimes Y$ space, then $Z$ is also a proper $G$-space). Hence, by Theorem \ref{thm:bccc} it maps the localizing subcategory generated by the projective objects in $\KKK^{G\ltimes Y}$ to the corresponding localizing subcategory generated by projective objects in $\KKK^G$.

Then, since the Dirac morphism is determined uniquely up to isomorphism of the associated exact triangles, we may assume that the natural morphism $D_{C_0(Y)}\in \KKK^{G}(P(C_0(Y)),C_0(Y))$ is an isomorphism as well.  Let $\beta$ denote its inverse.
Then the composition $\eta:=P(p)^{-1}\circ \beta\circ p\in \KKK^G(C_0(G^{0}),P(C_0(G^{0})))$ is the desired morphism. The final assertion then follows from the commutative diagram (\ref{eq:bcmap}).
\end{proof}
An element $\eta$ as in the theorem above is often called a dual Dirac morphism for $G$ (see \cite[Definition~8.1]{meyernest:tri}) and is unique (if it exists).

\subsection{Permanence properties}
In this section we will often need to compare the subcategories $\CI$ and $\mathcal{N}$ for different groupoids. To highlight this, we will slightly adjust our notation and write $\mathcal{N}_G$ for the weakly contractible objects in $\KKK^G$ and $\mathcal{CI}_G$ for the compactly induced objects. 

Sometimes we write ``BC'' as a shorthand for ``Baum--Connes conjecture''.

\subsubsection{Subgroupoids}
Given a second countable étale groupoid $G$, and a subgroupoid $H\subseteq G$ we may ask how the (strong) Baum--Connes conjectures for $G$ and $H$ are related. We need
\begin{lemm}\label{Lem:Res&Ind are nice functors}
Suppose $H\subseteq G$ is a subgroupoid. Then the following hold:
\begin{enumerate}
    \item If $H\subseteq G$ is open, then $\Res^H_G(\mathcal{N}_G)\subseteq \mathcal{N}_H$.
    \item If $H$ is closed in $G|_{H^0}$, then $\Res^H_G(\langle\mathcal{CI}_G\rangle)\subseteq \langle\mathcal{CI}_H\rangle$. \label{lem:res and ci}
    \item If $H$ is open in $G$ and closed in $G|_{H^0}$, then $\Res_G^H$ maps a Dirac triangle for $G$ to a Dirac triangle for $H$.
\end{enumerate}
\end{lemm}
\begin{proof}
To show the first item suppose $H$ is an open subgroupoid of $G$ and let $N\in \mathcal{N}_G\subseteq \KKK^G$. Suppose that $Q$ is a proper open subgroupoid of $H$. Then $Q$ is also a proper open subgroupoid of $G$ and hence $\Res_H^Q(\Res_G^H(\mathrm{id}_N))=\Res_G^Q(\mathrm{id}_N)\stackrel{\ref{lem:contractible objects}}{=}0$. Another application of Lemma \ref{lem:contractible objects} yields the result.

Next suppose $H$ is closed in $G|_{H^0}$. Whenever $G$ acts properly on a space $Z$ with anchor map $p:Z\rightarrow G^0$, then the action restricts to a proper action of $H$ on $p^{-1}(H^0)$. In particular, it follows that $\Res^H_G(\Cc\I_G)\subseteq \Pr_H$ and hence $\Res_G^H(\langle\Cc\I_G\rangle)\subseteq \langle\Cc\I_H\rangle$ by Theorem \ref{theo:ci=pr}.

The final assertion is a direct consequence of the first two statements.
\end{proof}

\begin{lemm}
Suppose $H\subseteq G$ is a subgroupoid such that $H$ is closed in $G|_{H^0}$. Then the following hold:
$$\Ind_H^G:\KKK^H\rightarrow \KKK^G$$ is triangulated, $\Ind_H^G(\mathcal{N}_H)\subseteq \mathcal{N}_G$, and  $\Ind_H^G\langle\Cc\I_H\rangle\subseteq\langle\Cc\I_G\rangle$. In particular, it maps Dirac triangles to Dirac triangles.
\end{lemm}
\begin{proof}
Induction in stages gives that a compactly induced object in $\KKK^H$ is mapped to a proper object in $\KKK^G$. Indeed, If $Q\subseteq H$ is a compact action, then $\Ind_H^G(\Ind_Q^H A)=\Ind_Q^G A$. It follows from our assumption that $Q$ is closed in $G|_{Q^0}$, and hence the action of $G$ on $G_{Q^0}/Q$ is proper. It follows immediately that $\Ind_Q^G A$ is a proper $G$-algebra (see also the induction picture in \cite{bon:goingdownbc}). Whence $\Ind_H^G\langle\Cc\I_H\rangle\subseteq\langle\Cc\I_G\rangle$ by Theorem \ref{theo:ci=pr}.

Finally, let $A\in \mathcal{N}_H\subseteq \KKK^H$. Then by Lemma \ref{Lem:Res&Ind are nice functors}.(2) we have $$\Res_G^H(P_G(C_0(G^0)))\otimes^{\text{max}}_{H^0}A\cong P_H(C_0(H^0))\otimes^{\text{max}}_{H^0}A\cong 0.$$
Using Lemma \ref{lem:indcomp} we conclude that
$$P_G(C_0(G^0))\otimes_{G^0}\Ind_H^G A\cong\Ind_H^G(\Res_G^H(P_G(C_0(G^0)))\otimes_{H^0}A)\cong 0$$ as well.
\end{proof}

The following result was already observed by Tu \cite{tu:coarsebcgII} for the classical Baum--Connes conjecture. Unfortunately, his proof relies on \cite[Lemma~3.9]{tu:coarsebcgII}, which seems to be erroneous. A counterexample where $G$ is the compact space $[0,1]$ (viewed as a trivial groupoid just consisting of units) is exhibited in \cite[Example~5.6]{dadmey:Etop} and \cite[p.36]{bauval:nuc}.
\begin{theo}\label{thm:BC passes to subgrpds}
Let $G$ be a second countable groupoid, $H\subseteq G$ be an \'etale subgroupoid that is closed in $G|_{H^0}$, and $A\in \KKK^H$. Then there is a natural $\KKK$-equivalence between $P_G(\Ind_H^G A)\rtimes_r G$ and $P_H(A)\rtimes_r H$.
Hence the (strong) Baum--Connes conjecture with coefficients passes to closed subgroupoids and restrictions to open subsets.
\end{theo}
\begin{proof}
From the previous lemma we conclude that
$P_G(\Ind_H^G A)\rtimes_r G\cong \Ind_H^G(P_H(A))\rtimes_r G$. The latter however is canonically Morita-equivalent (and hence in particular $\KKK$-equivalent) to $P_H(A)\rtimes_r H$. The result about the (strong) Baum--Connes conjecture follows readily.
\end{proof}
\subsubsection{Continuity in the coefficient algebra}
Let $(A_n)_n$ be an inductive system of $G$-$C^*$algebras and let $A=\varinjlim A_n$ be the inductive limit. In \cite[Section~3]{bondel:kun} it was shown that $A$ carries a canonical $G$-action making all the structure maps equivariant, i.e., the inductive limit exists in the category of $G$-$C^*$-algebras.

\begin{prop}\label{prop:limits}
Let $(A_n)_n$ be an admissible inductive system of $G$-$C^*$algebras and let $A=\varinjlim A_n$.
Then $P(A)\rtimes_r G$ is naturally $\KKK$-equivalent to $\hlim (P(A_n)\rtimes_r G)$, and $N(A)\rtimes_r G$ is naturally $\KKK$-equivalent to $\hlim (N(A_n)\rtimes_r G)$.

If furthermore $G$ satisfies the (strong) Baum--Connes conjecture with coefficients in $A_n$ for all $n\in\N$, then $G$ satisfies the (strong) Baum--Connes conjecture with coefficients in $A$.
\end{prop}
\begin{proof}
Let us consider the following diagram,
\begin{center}
    
\begin{tikzcd}
\bigoplus_n P(A_n) \arrow[r, "\cong"] \arrow[d, "id-S"] & P(\bigoplus_n A_n) \arrow[r, "D_{\oplus A_n}"] \arrow[d, "P(id-S)"] & \bigoplus_n A_n \arrow[d, "id-S"] \\
\bigoplus_n P(A_n) \arrow[r, "\cong"]                   & P(\bigoplus_n A_n) \arrow[r, "D_{\oplus A_n}"']                     & \bigoplus_n A_n                  
\end{tikzcd}.
\end{center}
The horizontal maps in the left-hand square are the natural isomorphisms obtained from the facts that the categories $\CI$ and $\mathcal{N}$ are closed under direct sums and the Dirac triangle is unique. The square on the right commutes by naturality of the Dirac morphism. By \cite[Proposition~1.1.11]{bbd:faisceauxpervers} the outer square forms the center of a larger diagram, in which each row and column is an exact triangle, and each square commutes (up to a sign), as shown below.

\begin{center}
  
\begin{tikzcd}
\Sigma \hlim N(A_n) \arrow[d] \arrow[r, dashed] & \Sigma \hlim P(A_n) \arrow[d] \arrow[r, dashed] & \Sigma A \arrow[d] \arrow[r, dashed] & \Sigma \hlim N(A_n) \arrow[d] \\
\bigoplus_n \Sigma N(A_n) \arrow[d] \arrow[r]            & \bigoplus_n P(A_n) \arrow[r] \arrow[d]                   & \bigoplus_n A_n \arrow[d] \arrow[r]  & \bigoplus N(A_n) \arrow[d]             \\
\bigoplus_n \Sigma N(A_n) \arrow[d] \arrow[r]            & \bigoplus_n P(A_n) \arrow[r] \arrow[d]                   & \bigoplus_n A_n \arrow[d] \arrow[r]  & \bigoplus N(A_n) \arrow[d, dashed]     \\
\hlim \Sigma N(A_n) \arrow[r, dashed]           & \hlim P(A_n) \arrow[r, dashed]                  & A \arrow[r, dashed]                  & \hlim N(A_n)                 
\end{tikzcd}
\end{center}

Since the horizontal maps in the middle square are the morphisms defining the homotopy limit uniquely up to isomorphisms, it is clear which objects appear in the first and last row. In the diagram above we have already made use of the fact that the sequence $(A_n)_n$ is admissible by replacing $\hlim A_n$ by the inductive limit $A=\lim A_n$. Consider now the bottom row of the diagram. Since $\CI$ and $\mathcal{N}$ are localizing subcategories, they are closed under homotopy direct limits. Hence, by uniqueness, the bottom row is naturally isomorphic to the exact triangle
$$\Sigma N(A)\rightarrow P(A)\rightarrow A\rightarrow N(A).$$
Taking reduced crossed products is a triangulated functor on $\KKK^G$, so we can take crossed products throughout the diagram, completing the proof of the first assertion.

Now if $G$ satisfies the strong Baum--Connes conjecture with coefficients in $A_n$ for each $n$, then the horizontal arrows in the central square are $\KKK$-equivalences (after taking reduced crossed products). It then follows immediately that $\jmath^G_r(D_A)$ is also a $\KKK$-equivalence. For the classical version of the Baum--Connes conjecture first apply the reduced crossed product functor to the diagram above and then note that the two middle columns in the resulting diagram induce a homomorphism of long exact sequences in $K$-theory. An application of the Five-Lemma yields the result.
\end{proof}

\subsubsection{Products and unions of subgroupoids}
Let $G=\bigcup G_n$ be a union of a sequence of clopen subgroupoids. We shall need the $G_n$ to be open so that if $A\in \KKK^G$, we can write the crossed product as an inductive limit $A\rtimes_r G=\lim A\rtimes_r G_n$ as well. Since the $G_n$ are also closed we obtain canonical restriction maps $\Gamma_c(G,\mathcal{A})\rightarrow \Gamma_c(G_n,\mathcal{A})$, which induce completely positive contractions $A\rtimes_r G\rightarrow A\rtimes_r G_n$. It follows that the inductive system $(A\rtimes_r G_n)_n$ is admissible and hence in the category $\KKK$ we can identify the direct limit $A\rtimes_r G$ with the homotopy direct limit $\text{ho-lim}\,A \rtimes_r G_n$.

\begin{prop}
Let $(G_n)_n$ be a sequence of clopen subgroupoids of $G$ such that $G=\bigcup_n G_n$. Suppose $A\in \KKK^G$ such that $G_n$ satisfies (strong) BC with coefficients in $\Res_G^{G_n}(A)$ for all $n\in\N$. Then $G$ satisfies (strong) BC with coefficients in $A$.
\end{prop}
\begin{proof}
We know from Lemma \ref{Lem:Res&Ind are nice functors} that $\Res_G^{G_n}$ preserves Dirac triangles. It follows that in $\KKK$ we have identifications
$$P_{G_n}(\Res_G^{G_n}(A))\rtimes_r G_n\cong (\Res_G^{G_n} P(A))\rtimes_r G_n,$$
and similarly 
$$N_{G_n}(\Res_G^{G_n}(A))\rtimes_r G_n\cong (\Res_G^{G_n} N(A))\rtimes_r G_n.$$
By taking limits we get
$$P(A)\rtimes_r G\cong \text{ho-lim}\,P(A)\rtimes_r G_n\cong \text{ho-lim}\,P_{G_n}(\Res_G^{G_n}(A))\rtimes_r G_n$$ and similarly
$$N(A)\rtimes_r G\cong \text{ho-lim}\,N_{G_n}(\Res_G^{G_n}(A))\rtimes_r G_n$$
Recall that $G$ satisfies the (strong) Baum--Connes conjecture with coefficients in $A$ if and only if $N(A)\rtimes_r G$ is $K$-contractible (or $\KKK$-contractible for the strong version). Since the categories of $K$-contractible (resp. $\KKK$-contractible) objects are localising, they are closed under homotopy direct limits. The result follows.
\end{proof}

Let us now turn our attention to direct products. Suppose $G=G_1\times G_2$ is the product of two \'etale groupoids $G_1,G_2$. Suppose further that $A_i\in \KKK^{G_i}$  for $i=1,2$. If either $A_1$ or $A_2$ is exact, the minimal tensor product $A:=A_1\otimes A_2$ comes equipped with a diagonal action and hence can be viewed as an object in $\KKK^G$.

\begin{prop}\label{Prop:Strong BC for products}
If $G_i$ satisfies strong BC with coefficients in $A_i$ for $i=1,2$, then $G_1\times G_2$ satisfies strong BC with coefficients in $A_1\otimes A_2$.
\end{prop}
\begin{proof}
We claim that
$\Cc\I_{G_1}\otimes \Cc\I_{G_2}\subseteq \Cc\I_{G_1\times G_2}$ and $\mathcal{N}_{G_1}\otimes\mathcal{N}_{G_2}\subseteq \mathcal{N}_{G_1\times G_2}$.
It follows in particular, that if $P_i\rightarrow C_0(G_i^0)\rightarrow N_i$ is a Dirac triangle for $G_i$, $i=1,2$, then
$$P_1\otimes P_2 \rightarrow C_0((G_1\times G_2)^0)\rightarrow N_1\otimes N_2$$
is a Dirac triangle for $G=G_1\times G_2$.
Since the minimal tensor product behaves well with respect to reduced crossed products, we have canonical isomorphisms
\begin{gather*}
A\rtimes_r G  \cong (A_1\rtimes_r G_1)\otimes(A_2\rtimes_r G_2)\\
P_G(A)\rtimes_r G \cong (P_{G_1}(A_1)\otimes P_{G_2}(A_2))\rtimes_r G  \cong (P_{G_1}(A_1)\rtimes_r G_1)\otimes (P_{G_2}(A_2)\rtimes_r G_2),
\end{gather*}
where the first $\KKK$-equivalence follows from the above observation about Dirac triangles.
Under these identifications, the Baum--Connes assembly map $P_G(A)\rtimes_r G\rightarrow A\rtimes_r G$ decomposes as the exterior tensor product of the Baum--Connes assembly maps $P_{G_i}(A_i)\rtimes_r G_i\rightarrow A_i\rtimes_r G_i$.
Since the exterior tensor product of $\KKK$-equivalences is a $\KKK$-equivalence itself, the result follows.
\end{proof}

As a an immediate consequence we have the following:
\begin{coro}
Let $A_1,A_2\in \KKK^G$  such that at least one of the two is exact. Then $A_1\otimes_{G^0}A_2\in \KKK^G$, where $\otimes_{G^0}$ denotes the balanced minimal tensor product. If we further assume that $G$ satisfies strong BC with coefficients in $A_1$ and $A_2$, then $G$ satisfies strong BC with coefficients in $A_1\otimes_{G^0}A_2$.
\end{coro}
\begin{proof}
Proposition \ref{Prop:Strong BC for products} implies that $G\times G$ satisfies the strong Baum--Connes conjecture with coefficients in $A_1\otimes A_2$. View $G$ as a closed subgroupoid of $G\times G$ via the diagonal inclusion. Since $\Res_{G\times G}^G(A_1\otimes A_2)\cong A_1\otimes_{G^0} A_2$, the result follows from Theorem \ref{thm:BC passes to subgrpds}.
\end{proof}
The corresponding results for the classical Baum--Connes conjecture require further assumptions, since the K\"unneth formula for the computation of the $K$-theory of a tensor product does not always hold. A detailed study in this direction has been carried out by Dell'Aiera and the first named author in \cite{bondel:kun}. 

Using the methods developed in the present article the results on the classical Baum--Connes conjecture with coefficients in a minimal balanced tensor product presented in \cite{bondel:kun} can be extended to all \'etale groupoids.

\subsection{Group bundles}
We can now strengthen the results on group bundles obtained in \cite{bon:goingdownbc}.
\begin{theo}
Let $G$ be a second countable étale group bundle which is strongly amenable at infinity. We suppose further that $G^0$ is locally finite-dimensional. Let $A$ be a separable $G$-algebra which is continuous as a field of $C^*$-algebras over $G^0$. If the discrete group $G_u^u$ satisfies BC with coefficients in $A_u$ for every $u\in G^0$, then $G$ satisfies BC with coefficients in $A$.
\end{theo}
\begin{proof}
We will first prove this in the case that $G^{0}$ is compact and finite dimensional. Since we are working with second countable compact Hausdorff spaces the covering dimension of $X$ coincides with the small inductive dimension of $X$, which we are going to employ.
The proof will proceed by induction on the dimension of $X$. The zero-dimensional case has already been considered in \cite[Theorem 8.11]{bon:goingdownbc}. Assume that $dim(X)=n$ and the result has been shown for all spaces of dimension strictly smaller than $n$.
It is enough to show $(1-\gamma_A)K_*(A\rtimes_r G)=\{0\}$. So let $x\in (1-\gamma_A)K_i(A\rtimes_rG)$. By our assumption that $G_u^u$ satisfies BC with coefficients in $A_u$ and \cite[Lemma 8.10]{bon:goingdownbc} we have $q_{u,*}(x)=0$ for all $u\in G^{0}$. Using \cite[Lemma~3.4]{cen:conneskasparovconj} we can find an open neighbourhood $U_u$ of $u$ in $G^{0}$ such that $q_{\overline{U_u},*}(x)=0$. Next, apply the fact that $G^{0}$ has inductive dimension at most $n$ to replace each of the sets $U_u$ by a smaller neighbourhood of $u$ to assume additionally, that $\dim(\overline{U_u}\setminus U_u)\leq n-1$. Using compactness of $G^{0}$ we may find a finite subcover say $U_1,\ldots, U_l$ such that $\dim(\overline{U_i}\setminus U_i)\leq n-1$ and $q_{\overline{U_i},*}(x)=0$ for all $1\leq i\leq l$. Consider the open set $O:=G^{0}\setminus \bigcup_{i=1}^l \partial U_i$ and the associated ideal $A_O:=C_0(O)A$. Then $C_0(O)(A\rtimes_r G)= A_O\rtimes_r G_O$. 
Since $G$ is exact we have a short exact sequence of $C^*$-algebras
$$0\rightarrow A_O\rtimes G_O\rightarrow A\rtimes_r G\rightarrow A_Y\rtimes_r G_Y\rightarrow 0.$$
We want to consider the induced $6$-term exact sequence in $K$-theory. Since the boundaries $\partial U_i$ are closed and at most $(n-1)$-dimensional, so is their union $Y:=\bigcup_{i=1}^l \partial U_i$. Applying the induction hypothesis yields that $(1-\gamma_{A_{Y}})K_*(A_{Y}\rtimes_r G_{Y})=0$.
Hence the $6$-term exact sequence in $K$-theory shows that the canonical inclusion map induces an isomorphism
$$(1-\gamma_{A_O})K_i(A_O\rtimes_r G_O)\cong (1-\gamma_A)K_i(A\rtimes_r G).$$
It follows that there exists a unique element $x'\in (1-\gamma_{A_O})K_i(A_O\rtimes_r G_O)$ whose image under the inclusion map is $x$.
Furthermore, $O$ can be decomposed as a finite disjoint union of open sets $O=\bigsqcup_{j=1}^m W_j$ such that each $W_j$ is contained in at least one of the sets $U_i$ by a standard inclusion/exclusion argument. Corresponding to this decomposition is a decomposition of the crossed product $A_O\rtimes_r G_O$ as
$$A_O\rtimes G_O= \bigoplus_{j=1}^m A_{W_j}\rtimes_r G_{W_j}.$$
It follows that $x'=\sum_{j=1}^l x_j'$ where $x_j'$ is in the image of the inclusion map $(1-\gamma_{A_{W_j}})K_i(A_{W_j}\rtimes G_{W_j})\rightarrow (1-\gamma_{A_O})K_i(A_O\rtimes_r G_O)$. Thus, it is enough to show that $x_j'=0$ for all $j=1,\ldots, l$. To this end consider the short exact sequence
\begin{equation}\label{grpbundlesseq}
    0\rightarrow A_{W_j}\rtimes G_{W_j}\rightarrow A_{\overline{W_j}}\rtimes_r G_{\overline{W_j}}\rightarrow A_{\partial W_j}\rtimes_r G_{\partial W_j}\rightarrow 0.
\end{equation}
Since $\partial W_j\subseteq \partial U_i$ is a closed subset for some $U_i$, the boundary of $W_j$ has dimension at most $n-1$. Hence we can apply the induction hypothesis again, to see that $(1-\gamma_{A_{\partial W_j}})K_*(A_{\partial W_j}\rtimes_r G_{\partial W_j})=0$. The $6$-term exact sequence in $K$-theory induced by (\ref{grpbundlesseq}) shows that the inclusion map induces an isomorphism
$(1-\gamma_{A_{W_j}})K_i(A_{W_j}\rtimes G_{W_j})\rightarrow (1-\gamma_{A_{\overline{W_j}}})K_i(A_{\overline{W_j}}\rtimes_r G_{\overline{W_j}}).$ The image of $x_j'$ under this map coincides with the image of $x$ under the restriction map $q_{\overline{W_j},*}$. Since  $W_j\subseteq U_i$ for some $1\leq i\leq n$ we get that $q_{\overline{W_j},*}(x)=q_{\overline{W_j},*}(q_{\overline{U_i},*}(x))=0$ and this completes the proof for compact and finite dimensional unit spaces.

Finally, if $G^0$ is a locally finite-dimensional and locally compact space, write $G^0$ as an increasing union $\bigcup K_n$ of compact subsets of $G^0$ such that $K_n\subseteq \mathrm{int}(K_{n+1})$. Using that $G^0$ is locally finite-dimensional, we may assume that each $K_n$ has finite dimension. The first part of this proof implies that $G|_{K_n}$ satisfies BC with coefficients in $A|_{K_n}$ and $G|_{\partial K_n}$ satisfies BC with coefficients in $A|_{\partial K_n}$. A $6$-term exact sequence argument (using exactness of $G$!) then shows that $G|_{\mathrm{int}(K_n)}$ satisfies BC with coefficients in $A|_{\mathrm{int}(K_n)}$ for all $n\in\N$.
Now we can write $A=\lim A|_{\mathrm{int}(K_n)}$. Picking an approximate unit $(\rho_n)_n$ with $\rho_n\in C_c(\mathrm{int}(K_n))$ we can define completely positive contractions $A\rightarrow A|_{\mathrm{int}(K_n)}$ by $a\mapsto \rho_n a$ which converge pointwise to the identity.  Hence the sequence  $A|_{\mathrm{int}(K_n)}$ is admissible and the result follows from Proposition \ref{prop:limits}.
\end{proof}

The class of infinite dimensional spaces to which the previous result applies includes all locally compact CW complexes. An example of a compact space that is not covered by the result is the Hilbert cube.

\backmatter

\bibliographystyle{amsplain-init-nodash}
\bibliography{BibliographyBST}                 

\providecommand{\bysame}{\leavevmode\hbox to3em{\hrulefill}\thinspace}
\providecommand{\MR}{\relax\ifhmode\unskip\space\fi MR }
\providecommand{\MRhref}[2]{%
  \href{http://www.ams.org/mathscinet-getitem?mr=#1}{#2}
}
\providecommand{\href}[2]{#2}
\begin{thebibliography}{10}

\bibitem{renroch:amgrp}
C.~Anantharaman-Delaroche and J.~Renault, \emph{Amenable groupoids}, Monographies de L'Enseignement Math\'ematique [Monographs of L'Enseignement Math\'ematique], vol.~36, L'Enseignement Math\'ematique, Geneva, 2000, With a foreword by Georges Skandalis and Appendix B by E. Germain. \MR{1799683}

\bibitem{ad:exact}
C.~Anantharaman-Delaroche, \emph{Exact groupoids}, 2021.

\bibitem{bali:CartanUCT}
S.~{Barlak} and X.~{Li}, \emph{{Cartan subalgebras and the UCT problem}}, {Adv. Math.} \textbf{316} (2017), 748--769 (English).

\bibitem{bauval:nuc}
A.~Bauval, \emph{{$RKK(X)$}-nucl{\'e}arit{\'e} (d'après {G}.\ {S}kandalis)}, $K$-Theory \textbf{13} (1998), no.~1, 23--40. \MR{1610242 (99h:19007)}

\bibitem{bbd:faisceauxpervers}
A.~A. {Beilinson}, J.~{Bernstein}, and P.~{Deligne}, \emph{{Faisceaux pervers}}, Analyse et topologie sur les espaces singuliers. CIRM, 6 - 10 juillet 1981. (Actes du Colloque de Luminy 1981). I, Astérisque, 1982 (French).

\bibitem{black:kth}
B.~Blackadar, \emph{{$K$}-theory for operator algebras}, second ed., Mathematical Sciences Research Institute Publications, vol.~5, Cambridge University Press, Cambridge, 1998. \MR{1656031}

\bibitem{bla:defhopf}
E.~Blanchard, \emph{D\'eformations de {$C^*$}-algèbres de {H}opf}, Bull. Soc. Math. France \textbf{124} (1996), no.~1, 141--215. \MR{1395009}

\bibitem{blakirch:glimm}
E.~{Blanchard} and E.~{Kirchberg}, \emph{{Global Glimm halving for \(C^*\)-bundles.}}, {J. Oper. Theory} \textbf{52} (2004), no.~2, 385--420 (English).

\bibitem{bon:goingdownbc}
C.~{B\"onicke}, \emph{{A going-down principle for ample groupoids and the Baum-Connes conjecture}}, {Adv. Math.} \textbf{372} (2020), 72 (English), Id/No 107314.

\bibitem{B21}
C.~B\"onicke, \emph{{K-theory and homotopies of twists on ample groupoids.}}, {J. Noncommut. Geom.} \textbf{15} (2021), no.~1, 195--222 (English).

\bibitem{bondel:kun}
C.~{B\"onicke} and C.~{Dell'Aiera}, \emph{{Going-down functors and the K\"unneth formula for crossed products by \'etale groupoids}}, {Trans. Am. Math. Soc.} \textbf{372} (2019), no.~11, 8159--8194 (English).

\bibitem{brown:proper}
J.~H. Brown, \emph{Proper actions of groupoids on {$C^*$}-algebras}, J. Operator Theory \textbf{67} (2012), no.~2, 437--467. \MR{2928324}

\bibitem{bdgw:matui}
C.~Bönicke, C.~Dell'Aiera, J.~Gabe, and R.~Willett, \emph{Dynamic asymptotic dimension and {M}atui's {HK} conjecture}, Proceedings of the London Mathematical Society \textbf{126} (2023), no.~4, 1182--1253.

\bibitem{ce:perma}
J.~{Chabert} and S.~{Echterhoff}, \emph{{Permanence properties of the Baum-Connes conjecture}}, {Doc. Math.} \textbf{6} (2001), 127--183 (English).

\bibitem{cen:conneskasparovconj}
J.~{Chabert}, S.~{Echterhoff}, and R.~{Nest}, \emph{{The Connes-Kasparov conjecture for almost connected groups and for liner \(p\)-adic groups}}, {Publ. Math., Inst. Hautes \'Etud. Sci.} \textbf{97} (2003), 239--278 (English).

\bibitem{ceo:shapiro}
J.~{Chabert}, S.~{Echterhoff}, and H.~{Oyono-Oyono}, \emph{{Shapiro's lemma for topological \(K\)-theory of groups}}, {Comment. Math. Helv.} \textbf{78} (2003), no.~1, 203--225 (English).

\bibitem{dadmey:Etop}
M.~{Dadarlat} and R.~{Meyer}, \emph{{E-theory for \(C^*\)-algebras over topological spaces.}}, {J. Funct. Anal.} \textbf{263} (2012), no.~1, 216--247 (English).

\bibitem{luckdavis:fjbc}
J.~F. {Davis} and W.~{L\"uck}, \emph{{Spaces over a category and assembly maps in isomorphism conjectures in \(K\)- and \(L\)-theory}}, {\(K\)-Theory} \textbf{15} (1998), no.~3, 201--252 (English).

\bibitem{hoyo:lie}
M.~L. del Hoyo, \emph{Lie groupoids and their orbispaces}, Portugalie Mathematica \textbf{70} (2013), no.~2, 161--209.

\bibitem{ivo:thesis}
I.~Dell'Ambrogio, \emph{Prime tensor ideals in some triangulated categories of {$C^*$}- algebras}, Ph.D. thesis, ETH Z\"urich, 2008.

\bibitem{meyereme:dualities}
H.~Emerson and R.~Meyer, \emph{Dualities in equivariant {K}asparov theory}, New York J. Math. \textbf{16} (2010), 245--313. \MR{2740579}

\bibitem{gill:twist}
E.~{Gillaspy}, \emph{{\(K\)-theory and homotopies of 2-cocycles on transformation groups}}, {J. Oper. Theory} \textbf{73} (2015), no.~2, 465--490 (English).

\bibitem{GWY23}
E.~Guentner, R.~Willett, and G.~Yu, \emph{Dynamical complexity and controlled operator {K}-theory}, Ast{\'e}risque, Paris: Soci{\'e}t{\'e} Math{\'e}matique de France (SMF), to appear.

\bibitem{higlafskan:bc}
N.~Higson, V.~Lafforgue, and G.~Skandalis, \emph{Counterexamples to the {B}aum-{C}onnes conjecture}, Geom. Funct. Anal. \textbf{12} (2002), no.~2, 330--354. \MR{1911663}

\bibitem{skan:crossinv}
M.~Khoshkam and G.~Skandalis, \emph{Crossed products of {$C^*$}-algebras by groupoids and inverse semigroups}, J. Operator Theory \textbf{51} (2004), no.~2, 255--279. \MR{2074181}

\bibitem{hen:triloc}
H.~Krause, \emph{Localization theory for triangulated categories}, Triangulated categories, London Math. Soc. Lecture Note Ser., vol. 375, Cambridge Univ. Press, Cambridge, 2010, pp.~161--235. \MR{2681709}

\bibitem{kum:diagonals}
A.~{Kumjian}, \emph{{On \(C^ *\)-diagonals}}, {Can. J. Math.} \textbf{38} (1986), 969--1008 (English).

\bibitem{kls:haagerup}
B.~K. Kwaśniewski, K.~Li, and A.~Skalski, \emph{The haagerup property for twisted groupoid dynamical systems}, Journal of Functional Analysis \textbf{283} (2022), no.~1, 109484.

\bibitem{laff:kkban}
V.~Lafforgue, \emph{{$K$}-th\'eorie bivariante pour les algèbres de {B}anach, groupoïdes et conjecture de {B}aum-{C}onnes. {A}vec un appendice d'{H}erv\'e {O}yono-{O}yono}, J. Inst. Math. Jussieu \textbf{6} (2007), no.~3, 415--451. \MR{2329760}

\bibitem{gall:kk}
P.-Y. Le~Gall, \emph{Th\'eorie de {K}asparov \'equivariante et groupoïdes. {I}}, $K$-Theory \textbf{16} (1999), no.~4, 361--390. \MR{1686846}

\bibitem{leGall:Survey}
P.-Y. Le~Gall, \emph{Groupoid {$C^*$}-algebras and operator {$K$}-theory}, Groupoids in analysis, geometry, and physics ({B}oulder, {CO}, 1999), Contemp. Math., vol. 282, Amer. Math. Soc., Providence, RI, 2001, pp.~137--145. \MR{1855247}

\bibitem{li:orbitequivalence}
X.~{Li}, \emph{{Continuous orbit equivalence rigidity}}, {Ergodic Theory Dyn. Syst.} \textbf{38} (2018), no.~4, 1543--1563 (English).

\bibitem{li:cartan}
X.~{Li}, \emph{{Every classifiable simple C\(^*\)-algebra has a Cartan subalgebra}}, {Invent. Math.} \textbf{219} (2020), no.~2, 653--699 (English).

\bibitem{macd:kk}
L.~E. MacDonald, \emph{Equivariant kk-theory for non-hausdorff groupoids}, Journal of Geometry and Physics \textbf{154} (2020), 103709.

\bibitem{matui:hk}
H.~{Matui}, \emph{{Homology and topological full groups of \'etale groupoids on totally disconnected spaces.}}, {Proc. Lond. Math. Soc. (3)} \textbf{104} (2012), no.~1, 27--56 (English).

\bibitem{matui:productSFT}
H.~{Matui}, \emph{{\'Etale groupoids arising from products of shifts of finite type}}, {Adv. Math.} \textbf{303} (2016), 502--548 (English).

\bibitem{meyer:genhom}
R.~Meyer, \emph{Equivariant {K}asparov theory and generalized homomorphisms}, $K$-Theory \textbf{21} (2000), no.~3, 201--228. \MR{1803228}

\bibitem{mey:catasp}
R.~{Meyer}, \emph{{Categorical aspects of bivariant K-theory.}}, {\(K\)-theory and noncommutative geometry. Proceedings of the ICM 2006 satellite conference, Valladolid, Spain, August 31--September 6, 2006}, Z\"urich: European Mathematical Society (EMS), 2008, pp.~1--39 (English).

\bibitem{meyer:tri}
R.~Meyer, \emph{Homological algebra in bivariant {$K$}-theory and other triangulated categories. {II}}, Tbil. Math. J. \textbf{1} (2008), 165--210. \MR{2563811}

\bibitem{nestmeyer:loc}
R.~Meyer and R.~Nest, \emph{The {B}aum-{C}onnes conjecture via localisation of categories}, Topology \textbf{45} (2006), no.~2, 209--259. \MR{2193334}

\bibitem{meyernest:tri}
R.~Meyer and R.~Nest, \emph{Homological algebra in bivariant {$K$}-theory and other triangulated categories. {I}}, Triangulated categories, London Math. Soc. Lecture Note Ser., vol. 375, Cambridge Univ. Press, Cambridge, 2010, pp.~236--289. \MR{2681710}

\bibitem{moer:orbi}
I.~Moerdijk and D.~A. Pronk, \emph{Orbifolds, sheaves and groupoids}, $K$-Theory \textbf{12} (1997), no.~1, 3--21. \MR{1466622}

\bibitem{murewi:morita}
P.~S. Muhly, J.~N. Renault, and D.~P. Williams, \emph{Equivalence and isomorphism for groupoid {$C^\ast$}-algebras}, J. Operator Theory \textbf{17} (1987), no.~1, 3--22. \MR{873460}

\bibitem{damu:renequiv}
P.~S. Muhly and D.~P. Williams, \emph{Renault's equivalence theorem for groupoid crossed products}, New York Journal of Mathematics. NYJM Monographs, vol.~3, State University of New York, University at Albany, Albany, NY, 2008. \MR{2547343}

\bibitem{nee:tri}
A.~Neeman, \emph{Triangulated categories}, Annals of Mathematics Studies, vol. 148, Princeton University Press, Princeton, NJ, 2001. \MR{1812507}

\bibitem{val:shi}
S.~Nishikawa and V.~Proietti, \emph{Groups with spanier--whitehead duality}, Annals of {$K$}-theory \textbf{5} (2020), no.~3, 465--500.

\bibitem{OO23}
H.~Oyono-Oyono, \emph{Groupoids decomposition, propagation and operator {{\(K\)}}-theory}, Groups Geom. Dyn. \textbf{17} (2023), no.~3, 751--804 (English).

\bibitem{patr:repEth}
E.~{Park} and J.~{Trout}, \emph{{Representable \(E\)-theory for \(C_0(X)\)-algebras}}, {J. Funct. Anal.} \textbf{177} (2000), no.~1, 178--202 (English).

\bibitem{valmak:groupoid}
V.~Proietti and M.~Yamashita, \emph{{Homology and K-theory of dynamical systems I. Torsion-free ample groupoids}}, Ergodic Theory and Dynamical Systems (online first, doi:10.1017/etds.2021.50), 2021.

\bibitem{valmak:groupoidthree}
V.~Proietti and M.~Yamashita, \emph{{Homology and K-theory of dynamical systems. III. Beyond stably disconnected Smale spaces}}, {arXiv:2207.03118}, 2022.

\bibitem{valmak:groupoidtwo}
V.~Proietti and M.~Yamashita, \emph{{Homology and K-theory of dynamical systems. II. Smale spaces with totally disconnected transversal}}, Journal of Noncommutative Geometry \textbf{17} (2023), no.~3, 957–998.

\bibitem{put:HoSmale}
I.~F. Putnam, \emph{A homology theory for {S}male spaces}, Mem. Amer. Math. Soc. \textbf{232} (2014), no.~1094, viii+122. \MR{3243636}

\bibitem{ren:group}
J.~Renault, \emph{A groupoid approach to {$C^{\ast} $}-algebras}, Lecture Notes in Mathematics, vol. 793, Springer, Berlin, 1980. \MR{584266}

\bibitem{ren:cartan}
J.~Renault, \emph{Cartan subalgebras in {$C^*$}-algebras}, Irish Math. Soc. Bull. (2008), no.~61, 29--63. \MR{2460017}

\bibitem{rosscho:kunneth}
J.~Rosenberg and C.~Schochet, \emph{The {K}\"{u}nneth theorem and the universal coefficient theorem for {K}asparov's generalized {$K$}-functor}, Duke Math. J. \textbf{55} (1987), no.~2, 431--474. \MR{894590}

\bibitem{thomsen:univ}
K.~Thomsen, \emph{The universal property of equivariant {$KK$}-theory}, J. Reine Angew. Math. \textbf{504} (1998), 55--71. \MR{1656818}

\bibitem{tu:coarsebcgII}
J.~L. {Tu}, \emph{{The coarse Baum-Connes conjecture and groupoids. II}}, {New York J. Math.} \textbf{18} (2012), 1--27 (English).

\bibitem{tu:moy}
J.-L. Tu, \emph{La conjecture de {B}aum-{C}onnes pour les feuilletages moyennables}, $K$-Theory \textbf{17} (1999), no.~3, 215--264. \MR{1703305}

\bibitem{tu:novikov}
J.-L. Tu, \emph{La conjecture de {N}ovikov pour les feuilletages hyperboliques}, {$K$}-Theory \textbf{16} (1999), 129--184.

\bibitem{tu:propnonhaus}
J.-L. Tu, \emph{Non-{H}ausdorff groupoids, proper actions and {$K$}-theory}, Doc. Math. \textbf{9} (2004), 565--597. \MR{2117427}

\bibitem{erpwill}
E.~{van Erp} and D.~P. {Williams}, \emph{{Groupoid crossed products of continuous-trace \(C^\ast\)-algebras}}, {J. Oper. Theory} \textbf{72} (2014), no.~2, 557--576 (English).

\bibitem{ver:catder}
J.-L. Verdier, \emph{Des cat\'egories d\'eriv\'ees des cat\'egories ab\'eliennes}, Ast\'erisque (1996), no.~239, xii+253 pp. (1997), With a preface by Luc Illusie, Edited and with a note by Georges Maltsiniotis. \MR{1453167}

\bibitem{WY23}
R.~Willett and G.~Yu, \emph{{T}he {U}niversal {C}oefficient {T}heorem for {$C^*$}-{A}lgebras with {F}inite {C}omplexity}, Mem. Eur. Math. Soc., Berlin: European Mathematical Society, to appear (English).

\end{thebibliography}
\end{document}